\documentclass[preprint, 12pt]{elsarticle}

\journal{Neurocomputing}

\usepackage{amsmath,amssymb}
\usepackage{graphicx}
\usepackage{hyperref}
\usepackage{booktabs} 
\usepackage{algorithm,algpseudocode} 
\usepackage{caption}
\usepackage{xcolor}
\usepackage{caption}
\usepackage{subcaption}
\usepackage{float}
\usepackage{chngcntr}
\counterwithin{figure}{section}
\usepackage{comment}

\newcommand\mypar[1]{\medskip\noindent\textbf{#1}}

\usepackage{amsthm}
\newtheorem{theorem}{Theorem}[section]
\newtheorem{definition}[theorem]{Definition}

\newtheorem{remark}[theorem]{Remark}
\newtheorem{conjecture}[theorem]{Conjecture}
 
\def\td{\tilde{D}}
\newcommand\rev[1]{{{#1}}}

\begin{document}

\title{Worst-case convergence analysis of \\ relatively inexact gradient descent \\on smooth convex functions}
\author[inst1]{Pierre Vernimmen\corref{cor1}\fnref{fn1}}\author[inst1,inst2]{François Glineur}
\affiliation[inst1]{organization={UCLouvain, ICTEAM (INMA)}, 
            addressline={Avenue Georges Lemaître 4}, 
            city={Louvain-la-Neuve}, 
            postcode={1348}, 
            country={Belgium}}
\affiliation[inst2]{organization={UCLouvain, CORE}, 
               addressline={Voie du Roman Pays 34},
               city={Louvain-la-Neuve}, 
               postcode={1348}, 
               country={Belgium}}
\cortext[cor1]{Corresponding author. Email: pierre.vernimmen@uclouvain.be}
\fntext[fn1]{The first author is funded by the UCLouvain \emph{Fonds spéciaux de Recherche} (FSR).}

\begin{abstract}
We consider the classical gradient descent algorithm with constant stepsizes, where some error is introduced in the computation of each gradient. More specifically, we assume some relative bound on the inexactness in the sense that the norm of the difference between the true gradient and its approximate value is bounded by a certain fraction of the gradient norm. This paper presents a worst-case convergence analysis of this so-called relatively inexact gradient descent on smooth convex functions, using the Performance Estimation Problem (PEP) framework. We first derive the exact worst-case behavior of the method after one step. Then we study the case of several steps and provide computable upper and lower bounds using the PEP framework. Finally, we discuss the optimal choice of constant stepsize according to the obtained worst-case convergence rates.
\end{abstract}


\maketitle
\clearpage

\tableofcontents

\section{Introduction}\label{sec::Introduction}

\rev{\subsection{Problem Setup}

We consider the standard optimization problem of unconstrained minimization of a differentiable function:
\[
    \min_{x \in \mathbb{R}^d} f(x),
\]
where \(f : \mathbb{R}^d \to \mathbb{R}\) is assumed to be convex and \(L\)-smooth, that is
\[
    f(y) \geq f(x) + \langle \nabla f(x), y - x \rangle, \quad \forall x,y \in \mathbb{R}^d,
\]
and
\[
    \|\nabla f(x) - \nabla f(y)\| \leq L \|x-y\|, \quad \forall x,y \in \mathbb{R}^d
\]
(i.e.\@ its gradient is $L$-Lipschitz). We assume that \(f\) admits at least one minimizer, denoted by \(x^\star \in \arg\min f\).
The goal of this paper is to analyze first-order methods when encountering inexactness, and in particular Inexact Gradient Descent, for solving this problem.

\medskip
\noindent
\textbf{Remark.} The step-size in gradient-based methods typically depends on the smoothness constant \(L\).
In many applications, this constant can be computed exactly (for example when the Hessian is available),
while in others it can be estimated or bounded (see e.g. \cite{fazlyab2019efficient}).
Throughout this paper, we assume that an admissible value of the smoothness constant \(L\) is available.}

\subsection{First-order methods, relative inexactness and worst-case analysis}

\mypar{First-order methods and their practical relevance.}
First-order optimization methods, and in particular gradient-based techniques, are widely used in modern large-scale optimization. Their low per-iteration complexity and favorable convergence guarantees make them especially well-suited to machine learning and signal processing applications \cite{lan2020first}. Among these methods, the gradient descent with a constant stepsize stands out due to its simplicity, robustness, and effectiveness, in particular for smooth convex problems \cite{nesterov2013introductory}.

\mypar{Motivations for inexact gradient computations.}
In many practical situations, exact gradient computations are not available or too costly. For example, inexact gradients naturally arise when computations are performed with limited precision, as in floating point arithmetic, or when the gradient is computed approximately, e.g., through sampling, simulations, or inner optimization loops. These sources of inexactness can significantly affect the performance of gradient-based methods and are thus central to both theoretical analysis and practical algorithm design.

Inexactness may also be imposed intentionally to save computational resources, for example, by quantizing gradients or using less expensive gradient surrogates. In large-scale machine learning systems or distributed environments, robustness to such approximations becomes crucial, and understanding their impact on convergence is essential for reliable algorithmic implementation. This is especially true in modern hardware environments, where low-precision computations are often preferred to accelerate training or inference (see e.g. the discussion in \cite{vernimmen2025empirical}).

\mypar{Different models of inexactness.}
Several models of inexactness have been studied in the literature. A common approach assumes that the approximate gradient differs from the true one by an error with bounded norm, leading to an \emph{absolute} error model, as studied, for instance, in \cite{doi:10.1137/060676386}. In \cite{devolder2014first}, another notion of inexact gradient is developed, based on the maximal error incurred by the corresponding quadratic upper bound. In these frameworks, the norm of the additive error is independent of the magnitude of the gradient.

More recently, a \emph{relative} model of inexactness has attracted attention, in which the norm of the gradient error is bounded proportionally to the true gradient norm. Mathematically, let $d_k$ be the approximate gradient at $x_k$, it is then said to be a $\delta$-relatively inexact gradient if $\|d_k - \nabla f(x_k)\|\leq \delta \|\nabla f(x_k)\|$, where $\delta \in [0,1)$ is a parameter controlling inexactness. This model is particularly relevant when the gradient norm itself indicates proximity to optimality: more accurate gradients are naturally required near the solution. It also includes, as special cases, errors arising from quantization or low-precision computation, as studied in \cite{kempke2025low} in the context of mixed-integer linear programming. The relative error model thus often offers a more flexible and realistic description of inexactness in practical applications.

\mypar{Previous work.}  
This notion of relative inexactness is used in  \cite{de2017worst, de2020worst} in the context of convergence analysis, where the Performance Estimation Problem (PEP) framework \cite{drori2014performance, taylor2017smooth} is used to provide tight convergence guarantees for the inexact gradient descent on \(L\)-smooth, \(\mu\)-strongly convex functions. For example, \cite[Theorem 5.3]{de2020worst} establishes a tight linear convergence rate for gradient descent with constant stepsize \(h\) and relative inexactness \(\delta < \frac{2\mu}{L+\mu}\). This analysis is however restricted to the strongly convex setting, and the results become ineffective as \(\mu \to 0\), for the smooth convex case.

Subsequent works, such as \cite{vasin2024gradient}, \cite{vasin2023accelerated}, \cite{kornilov2023intermediate}, and \cite{vasin2025solving} extend the study of relative inexactness to several first-order and accelerated methods using Lyapunov-based analyses. In particular, \cite{vasin2024gradient} shows that the relatively inexact gradient descent applied to smooth convex functions converges to a minimizer for sufficiently small constant stepsizes. While the approaches from the above papers cover more algorithmic variants, they typically yield conservative and sometimes loose upper bounds (both for rates and allowed stepsizes), which may limit their usefulness.

Other related research with relative errors includes probabilistic relative error models \cite{hallakstudy}, momentum-based schemes interpreted as inexact gradient descent \cite{khanh2025convergence}, proximal methods under relative gradient errors \cite{bello2025relative}, robustness to noise in distributed settings \cite{wang2024robust}, and heuristic stopping criteria under gradient noise \cite{vasin2021stopping}. Complementary to these, \cite{thomsen2025tight} analyzes error-feedback schemes with compressed gradients, providing tight worst-case guarantees.

\subsection{Contributions}
Our main contribution is to extend the PEP-based analysis of \cite{de2020worst} beyond the strongly convex case to smooth convex functions, providing some tight convergence guarantees for relatively inexact gradient descent in this broader setting. This journal article is an extension of our conference paper \cite{vernimmen2024convergence}, which initiated the study of relatively inexact gradient descent in the convex setting using the PEP methodology, but left the analysis incomplete. Numerical experiments are presented in a separate companion paper \cite{vernimmen2025empirical}, by the same authors.

\mypar{Relatively inexact gradient descent.} More specifically, in this paper, we will provide an improved, and, in some cases, tight convergence analysis of the following relatively inexact gradient descent:

\begin{algorithm}[H]
\caption{Relatively Inexact Gradient Descent with constant (normalized) stepsize $h$}
\label{algo::inexact_gradient}
\begin{algorithmic}[1]
\State Given an $L$-smooth convex function $f(\cdot)$, a starting iterate $x_0$, a stepsize $h$ and a inexactness level $\delta \in [0,1)$
\State $k \gets 0$
\While{$k < \text{max\_iterations}$}
    \State Let $d_k$ be an approximate gradient at $x_k$ with $\delta$ relative inexactness, i.e. such that
    \begin{equation} \label{ine} \|d_k - \nabla f(x_k)\|\leq \delta \|\nabla f(x_k)\| \end{equation}
    \State Compute the next iterate as $x_{k+1} = x_k - \frac{h}{L}d_k$
    \State $k \gets k + 1$
\EndWhile
\end{algorithmic}
\end{algorithm}

Our analysis demonstrates that Algorithm \ref{algo::inexact_gradient} enjoys guaranteed convergence towards a minimizer of the function for any level of relative inexactness $\delta \in [0,1)$, provided the stepsize is well chosen. This result is significant, as it shows that gradient descent can be robust to large relative inexactness levels. Specifically, we show that the method converges for any stepsize $h\in [0,h_{\text{max}}]$, with $h_{\text{max}}=\frac{2}{1+\delta}$, and is no longer convergent for $h > h_\text{max}$. This tight bound on acceptable stepsizes improves the previously known result from \cite{vasin2024gradient}.

\mypar{Type of convergence analysis.} In the previously studied smooth strongly convex setting, linear convergence rates were established on the objective accuracy, the gradient norm, and the distance to the solution (see Theorems 5.1 and 5.4 in \cite{de2020worst}). However, in the smooth convex case, it is well known that linear convergence cannot hold. Instead, one has to resort to studying convergence rates based on two distinct criteria, characterizing respectively the starting iterate and the last iterate. The perhaps most classical type of rate in the smooth convex setting bounds the objective accuracy in terms of the initial distance to the solution, see e.g. \cite[Corollary 2.1.2]{nesterov2013introductory}.

\noindent However, in this paper, we will prove results of the following type:
\begin{equation}
\tfrac{1}{L} \| \nabla f(x_N) \|^2 \le \sigma_N \bigl( f(x_0) - f(x_*) \bigr).
\end{equation}
i.e., we bound the (squared) gradient norm in terms of the initial objective accuracy. Constant $\sigma_N$ describes the convergence rate after $N$ iterations. 
This type of rate in the gradient norm has also been studied in the exact case, see e.g. \cite{kim2021optimizing}, and also potentially allows generalization to nonconvex settings \cite{rotaru2024exact}.

\mypar{Main theorem.} In the exact case ($\delta=0$), it is well-known that the convergence rate $\sigma_N$ for smooth convex functions has order $\mathcal{O}(\frac{1}{N})$. More precisely, the following tight rate is proved in \cite[Theorem 2.1]{rotaru2024exact} (see the beginning of Section \ref{sec::one_step} for a complete proof for one iteration):

\begin{theorem}[\cite{rotaru2024exact}, Theorem 2.1]\label{theo::rotaru_exact}
    Algorithm \ref{algo::inexact_gradient} applied to a convex $L$-smooth function $f$ with a constant stepsize $h\in[0,2]$ and an exact gradient ($\delta=0$), started from iterate $x_0$, generates iterates satisfying 
    \begin{equation*}
        \tfrac{1}{L}\|\nabla f(x_N)\|^2\leq \max\left\{\frac{1}{Nh+\tfrac{1}{2}}, 2(1-h)^{2N}\right\}\left(f(x_0)-f(x_\ast)\right)
    \end{equation*}
\end{theorem}

\noindent The main theoretical result in this paper is Theorem \ref{theorem::full_th_inexact_N_steps}, which extends the above to situations where the gradient is relatively inexact.

\begingroup
\renewcommand\thetheorem{\ref{theorem::full_th_inexact_N_steps}}
\begin{theorem}
    Algorithm \ref{algo::inexact_gradient} applied to a convex $L$-smooth function $f$ with an inexact gradient with relative inexactness $\delta\in(0,1)$, 
    started from iterate $x_0$ with 
    a constant stepsize $h\in[0,\frac{2}{1+\delta}]$,  generates iterates satisfying 
    \[
        \tfrac{1}{L} \min_{k\in\{1,\cdots,N\}}\|\nabla f(x_k)\|^2 \leq \tilde{C}_N(h, \delta) (f(x_0) - f(x_\ast)),
    \]
    where $\tilde{C}_N(h, \delta)$ is given by the continuous function
    \[
        \tilde{C}_N(h, \delta) =
        \begin{cases}
            \frac{1}{Nh(1-\delta)+\tfrac{1}{2}}, & \text{if } h \in \left[0, \frac{3}{2(1+\delta)}\right), \\\frac{2\tilde{\lambda}}{N(h\tilde{\lambda}^2+2(h-1)\tilde{\lambda}+h-1)+\tilde{\lambda}}, & \text{if } h \in \left[\frac{3}{2(1+\delta)}, \frac{3\delta+2-\sqrt{4-3\delta^2}}{2\delta(\delta+1)}\right], \\
            \frac{2}{N((1-h(1+\delta))^{-2})-(N-1)}, & \text{if } h \in \left(\frac{3\delta+2-\sqrt{4-3\delta^2}}{2\delta(\delta+1)}, \frac{2}{1+\delta}\right],
        \end{cases}
    \]
    and $\tilde{\lambda}$ is the solution of a cubic equation (see Definition \ref{def::tilde_lambda}).
\end{theorem}
\addtocounter{theorem}{-1}
\endgroup

\noindent Convergence is hence governed by three distinct regimes depending on the stepsize, compared to only two regimes in the exact case. We also discuss situations where this rate is tight, which includes the case of a single step for any stepsize (see in Section \ref{sec::one_step}, where both analytical and numerical arguments are used to establish tightness), and some subset of the three stepsize regimes for any number of iterations (see Section \ref{sec::several_steps}).

\subsection{Performance Estimation}

\mypar{Formulation.} The tight convergence rates described in this paper were obtained using the Performance Estimation Problem (PEP) methodology, first introduced in \cite{drori2014performance}, which aims at computing the worst-case convergence rate of some optimization method as the solution of an optimization problem itself. 
More precisely, the convergence rate $\sigma_N$ after $N$ iterations of a given algorithm $\mathcal{A}$ over a given class of objective functions $\mathcal{F}$ can be shown to be equal to the optimal value of the following
Performance Estimation Problem 
\begin{equation*}
    \begin{aligned}\label{eq::reformulatedPEP_opti_problem}
    \sigma_N  = \max_{\{x_i,g_i,f_i\}_{i \in I}} \quad & \|g_N\|^2\\[-.3cm]
    & x_{1}, \ldots, x_N \textrm{ are generated from } x_0 \textrm{ by algorithm } \mathcal{A}, 
    \\ & \text{there exists an interpolating function $f \in \mathcal{F}$ such}
    \\ & \text{that $f(x_i) = f_i$ and $\nabla f(x_i) = g_i$ for all $i \in I$,}
    \\ & g_\ast = 0, \text{ and } f_0 - f_\ast  \leq 1.
    \end{aligned}
\end{equation*}
(with set $I = \{0, 1, \cdots, N, *\}$), which can be formulated exactly and solved as a tractable semidefinite optimization problem, via the use of a Gram matrix lifting \cite{taylor2017smooth}. 

The condition that guarantees that the iterates, gradients, and function values contained in variables $\{x_i,g_i,f_i\}_{i \in I}$ match some interpolating function $f \in \mathcal{F}$ must be made explicit. For the class of $L$-smooth convex functions studied in this paper, it was shown in \cite{taylor2017smooth} that this is equivalent to requiring the following inequality,  called \textit{interpolation condition}
\begin{equation}\label{eq::Interpolation_ineq}\tag{INT}
    Q_{ij} : \quad f_i \geq f_j + \langle g_j, x_i-x_j\rangle  + \tfrac{1}{2L}\|g_j-g_i\|^2 
\end{equation}
to hold for every pair of indices $i,j \in  I$. 

To analyze algorithms that use relatively inexact gradients, such as Algorithm \ref{algo::inexact_gradient}, additional variables $d_i$ are introduced and the defining inequality \eqref{ine} is added to the PEP, rewritten as the convex quadratic constraint $\| d_i - g_i\|^2 \le \delta^2 \| g_i\|^2$ (see \cite{taylor2017exact} for the idea of using inexact steps in a PEP).

The solution to the resulting semidefinite optimization problem can be obtained numerically using a standard solver such as MOSEK \cite{mosek}, used for all numerical results in this work. Several toolboxes are available to assist practitioners in deriving the formulation for a large variety of classes of algorithms and functions, such as PESTO \cite{taylor2017performance} and PEPit \cite{goujaud2024pepit}.

\mypar{Example.} To illustrate the Performance Estimation technique and give an example of its relevance to analyze inexact methods, we plot in Figure \ref{fig:exact_inexact_first_example} the tight numerical convergence rate obtained from solving the PEP for Algorithm \ref{algo::inexact_gradient} on $1$-smooth convex functions, showing that the behavior of a first-order method can be very different when relying on an inexact gradient (here we test $\delta =0.2$). On the left, the behavior of gradient descent with constant stepsize $h=1$ is barely modified. On the right, the same gradient descent with constant stepsize $h=1.65$, which improves the convergence rate in the exact case, is seen to deteriorate it significantly in the inexact case.

\begin{figure}[H]
    \centering
    \begin{subfigure}{0.49\textwidth}
        \includegraphics[width=\textwidth]{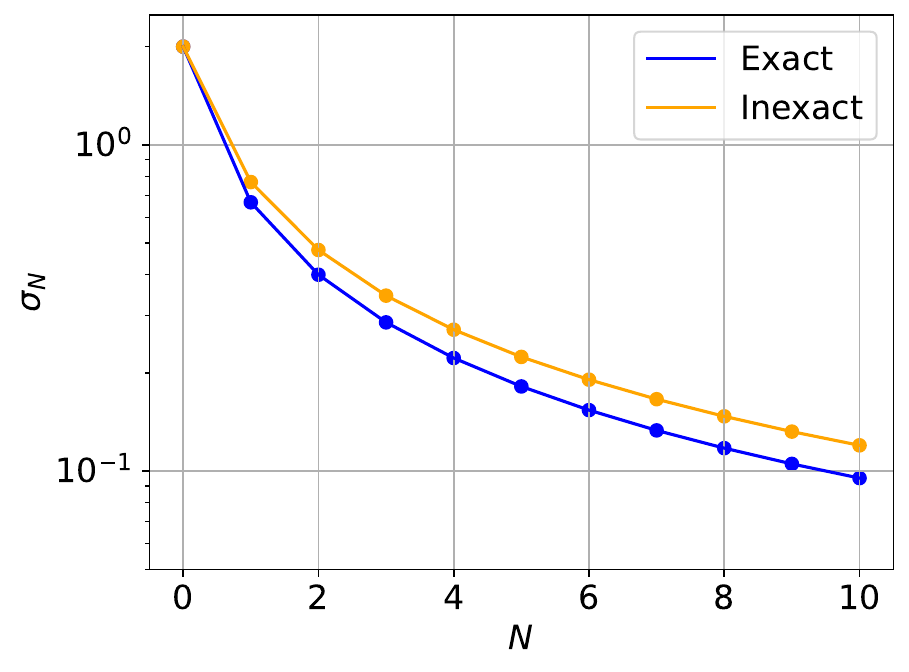}
        \caption{\centering $h=1$}
    \end{subfigure}
    \begin{subfigure}{0.49\textwidth}
        \includegraphics[width=\textwidth]{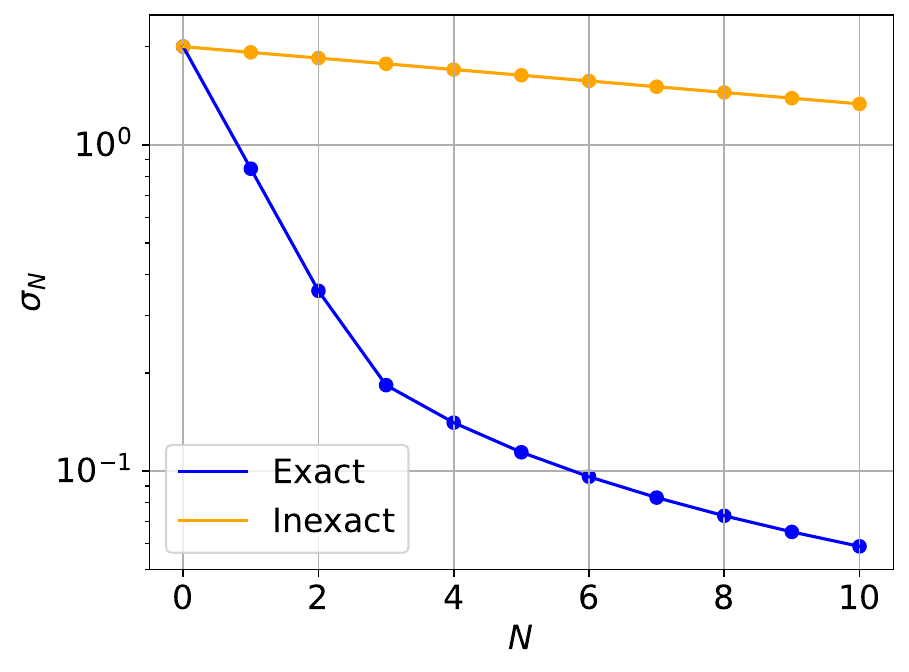}
        \caption{\centering $h=1.65$}
    \end{subfigure}
    \caption{\centering Convergence rate $\sigma_N$ of exact and inexact ($\delta=0.2$) gradient descent  after $N$ iterations.}
    \label{fig:exact_inexact_first_example}
\end{figure}

\mypar{Deriving proofs.} {Solving the PEP not only yields a tight worst-case performance guarantee, but also provides a way to derive a proof of those convergence bounds \cite{taylor2017smooth}. Those proofs are a byproduct of standard convex optimization duality. More concretely, all PEP convergence bounds can be obtained by summing the interpolation and inexactness inequalities arising in the formulation with suitable multipliers, which are given by optimal dual variables. 

In this work, we use the PEP framework not only to compute tight worst-case rates numerically, but also as an exploration tool to gain insight about the structure of valid convergence proofs. In several situations, we can analyze the numerical optimal value and solution (both primal and dual) of the PEP for several values of $N$ (number of iterations), $h$ (stepsize), and $\delta$  (inexactness level), identify patterns, and conjecture candidate analytical expressions for them. These serve as a basis to derive rigorous mathematical proofs for lower and upper bounds on the convergence rates.
This two-step approach, based on first gaining insights from numerical values and then producing a fully analytical argument, is classic within the PEP community, and allowed us to derive all convergence rates and proofs presented in this paper. 
\newline

\subsection{Paper organization}

In Section \ref{sec::one_step} we derive a tight worst-case rate after one step, i.e., find an expression for $\sigma_1$ that cannot be improved, as it is exactly matched by some function $f$. An explicit analytical expression for such a function $f$ is also provided for nearly all values of the stepsize.

Section \ref{sec::several_steps} studies the case where several steps of the methods are performed, and provides computable upper and lower bounds using the PEP framework. While performing a tight analysis becomes intractable beyond a few steps due to the symbolic complexity of the resulting equations, the associated semidefinite programs can be explicitly constructed and solved for any number of iterations, allowing to obtain tight numerical performance bounds (albeit with a computational cost that grows with $N$).

We also provide principled recommendations for the optimal choice of a constant stepsize according to the worst-case rate: as opposed to the exact case analyzed in \cite{taylor2017smooth}, the optimal stepsize does not seem to be very sensitive to the number of iterations. This reflects a key difference induced by the relative inexactness model and highlights the importance of adapting algorithmic parameters to account for inexactness. Section \ref{sec::conclusion} presents some conclusions. All figures from this paper, as well as the associated code, are openly available on \url{https://github.com/Verpierre/Inexact_gradient_descent}.

\section{One step of relatively inexact gradient descent}\label{sec::one_step}

\subsection{Convergence rate in the exact case ($\delta=0$)}

Before diving into the analysis of convergence rates in the relatively inexact setting, we first derive a comprehensive proof of one step of exact Gradient Descent ($N=1$ in  Theorem \ref{theo::rotaru_exact}). This result is proved in \cite{rotaru2024exact}, but we present an alternate proof here, whose structure will also be used in the inexact case.

\begin{theorem}[Convergence rate of one step of exact gradient descent on smooth convex functions]\label{th::exact_convergence_rate}
 One step of Algorithm \ref{algo::inexact_gradient} applied to a convex $L$-smooth function $f$ with an exact gradient ($\delta=0$), started from iterate $x_0$ with a constant stepsize $h\in[0,2]$, generates an iterate $x_1$ satisfying
\begin{equation*}
    \tfrac{1}{L} \|\nabla f(x_1)\|^2\leq \max \left\{ \frac{1}{h},\frac{2}{(1-h)^{-2}-1} \right\} (f(x_0)-f(x_1))
\end{equation*}

and

\begin{equation*}
        \tfrac{1}{L}\|\nabla f(x_1)\|^2\leq \max\left(\frac{1}{h+\tfrac{1}{2}}, 2(1-h)^{2}\right)\left(f(x_0)-f(x_\ast)\right)
\end{equation*}
    
\end{theorem}

\begin{proof}
    We assume without loss of generality that the smoothness constant is \(L = 1\), because of a standard homogeneity argument \cite[Section 3.5]{taylor2017smooth}. 
    A single step of the method computes \(x_1 = x_0 - h g_0\).
    Using interpolation inequalities \eqref{eq::Interpolation_ineq} derived from smoothness and convexity, we obtain
\begin{equation*}
\begin{aligned}
    Q_{01} &: f_0 - f_1 \geq \frac{g_0^2}{2} + \frac{g_1^2}{2} +(h-1) g_0 g_1 \\
    Q_{10} &: f_1 - f_0 \geq \frac{1-2h}{2} g_0^2 + \frac{g_1^2}{2} - g_0 g_1,
\end{aligned}
\end{equation*}
where the definition of $x_1$ is used to express everything in terms of gradients.

    We now form a nonnegative linear combination of these inequalities using the multipliers
    \begin{equation*}
    \begin{aligned}
        \lambda_{01} &= \lambda + 1, \\
        \lambda_{10} &= \lambda.
    \end{aligned}
    \end{equation*}
    where $\lambda$ is a nonnegative parameter that will be chosen later. The combination \(\lambda_{01}Q_{01} + \lambda_{10}Q_{10}\) yields a quadratic inequality of the form
    \begin{equation*}
        f_0 - f_1 \geq \frac{1}{2}
        \begin{pmatrix}
            g_0 \\ g_1
        \end{pmatrix}^\top
        \begin{pmatrix}
            2(1-h)\lambda +1 & (h-2)\lambda - 1 + h \\
            (h-2)\lambda - 1 + h & 2\lambda +1
        \end{pmatrix}
        \begin{pmatrix}
            g_0 \\ g_1
        \end{pmatrix}
    \end{equation*}
(this specific parametrization with $\lambda \ge 0$ actually ensures that the left-hand side is equal to $f_0-f_1$). 
We aim to lower-bound \(f_0 - f_1\) by a multiple of \(\|g_1\|^2\), i.e., obtain \(f_0 - f_1 \geq \rho \|g_1\|^2\) for some value of $\rho$, as large as possible. To this end, we decompose the above matrix as
\begin{equation*}
    f_0-f_1 \geq \frac{1}{2}
    \begin{pmatrix}
        g_0 \\ g_1
    \end{pmatrix}
    \left[
    \underbrace{\begin{pmatrix}
        2(1-h) \lambda +1 & (h-2) \lambda - 1+h \\
        (h-2)\lambda - 1 + h & 2\lambda+1-2\rho
    \end{pmatrix}}_{A}
    +
    \begin{pmatrix}
        0 & 0\\
        0 & 2\rho
    \end{pmatrix}
    \right]
    \begin{pmatrix}
        g_0 & g_1
    \end{pmatrix}.
\end{equation*}
When matrix \(A\) above is positive semidefinite, inequality \(f_0 - f_1 \geq \rho \|g_1\|^2\) follows.
    To find the best possible rate \(\rho\), we aim to maximize it under the constraint that \(A\) is positive semidefinite.

    Matrix $A$ is positive semidefinite if and only if all its principal minors are non-negative. This yields the following optimization problem
    \begin{equation*}
    \begin{aligned}
        \max_{\lambda, \rho \geq 0} \quad&\rho \\
        \text{s.t.}\quad &A_{0,0}=2(1-h) \lambda +1\geq0 \\
        &\det(A)=2\rho(2(h-1)\lambda-1)+h(-h(1+\lambda)^2+4\lambda+2)\geq0
    \end{aligned}
\end{equation*}
which will, in principle, compute the best possible convergence rate. We now distinguish three cases based on the value of the stepsize \(h\).

\begin{description}
   \item[Case 1: $h=1$.]
   As $\lambda \geq0$, the first constraint is trivially satisfied, and the second one simplifies to
   \begin{equation*}
       2\rho \leq -\lambda^2+2\lambda+1.
   \end{equation*}
   The right-hand side is a quadratic in \(\lambda\), maximized for \(\lambda_\ast = 1\), yielding
        \[
            \rho_\ast = \frac{-1^2 + 2 \cdot 1 + 1}{2} = 1.
        \]

   \item[Case 2: $h< 1$.]
   The first constraint becomes
   \begin{equation}\label{case2:constraint1}
       \lambda\geq \frac{1}{2(h-1)},
   \end{equation}
   while the second one can be rewritten as
   \begin{equation*}
       \rho \leq \frac{h(h(1+\lambda)^2-4\lambda-2)}{2(2(h-1)\lambda-1)}.
   \end{equation*}
   We define
        \[
            g(\lambda) := \frac{h(h(1+\lambda)^2-4\lambda-2)}{2(2(h-1)\lambda-1)},
        \]
        so that the second constraint is \(\rho \leq g(\lambda)\).
   To maximize \(\rho\), we thus maximize \(g(\lambda)\) subject to the constraint \eqref{case2:constraint1}. Computing the first-order derivative of $g(\lambda)$ yields the first-order optimality condition.
   \begin{equation*}
       \frac{dg(\lambda)}{d\lambda}=\frac{2h^2(\lambda-1)\left((h-1)\lambda+h-2\right)}{(2(h-1)\lambda-1)^2}=0
   \end{equation*}
   whose roots are $\lambda=\frac{2-h}{h-1}$ and $\lambda=1$. Studying the sign of $\frac{dg(\lambda)}{d\lambda}$ tells us that $g(\lambda)$ is maximized for $\lambda=1$, which satisfies \eqref{case2:constraint1}, as $h<1$. It thus leads to $\lambda_\ast=1$ and $\rho_\ast=g(\lambda_\ast)=h$. We can note that the rate found in \emph{Case 1} also matches this expression.

   \item[Case 3: $h> 1$.] In this case, the first constraint becomes
   \begin{equation}\label{case3:constraint1}
       \lambda \leq \frac{1}{2(h-1)}
       \end{equation}
    which is the same constraint as in the previous case, but with the inequality reversed. The second constraint is the same as in the previous case, with the same function $g(\lambda)$.
    
    When \(h \in (1, \tfrac{3}{2}]\), we observe that \(\lambda = 1\) satisfies the constraint \eqref{case3:constraint1}, so the maximum is attained at \(\lambda_\ast = 1\), yielding \(\rho_\ast = h\).

        When \(h\in[\frac{3}{2},2]\), the maximum of \(g(\lambda)\) shifts to \(\lambda = \frac{2 - h}{h - 1}\), which still satisfies the constraint \eqref{case3:constraint1}. Plugging this into \(g\) gives
        \[
            \rho_\ast = \frac{(2 - h)h}{2(h - 1)^2} = \frac{1}{2}\left( (1 - h)^{-2} - 1 \right).
        \]
   
\end{description}

\noindent Combining all three cases, the optimal multiplier is
    \[
        \lambda_\ast = \begin{cases}
            1 & \text{if } h \leq \tfrac{3}{2}, \\
            \frac{2 - h}{h - 1} & \text{if } h \in [\tfrac{3}{2}, 2],
        \end{cases}
    \]
    leading to the following convergence rate, combining two regimes
    \begin{equation*}
        f_0 - f_1 \geq \min\left(h, \frac{1}{2} ((1 - h)^{-2} - 1)\right) \|g_1\|^2
    \end{equation*}
    which is equivalent to the first part of the theorem. Then, as the function $f$ is $1$-smooth and convex, we can write the following 
    \begin{equation*}
         f(x)-f(x_\ast) \ge \frac{1}{2} \|\nabla f(x)\|^2,
    \end{equation*}
which comes from the interpolation inequality \ref{ine} written between $x$ and a minimizer $x^*$ (where $\nabla f(x_\ast)=0$ holds). This inequality with $x=x_1$ gives $f_1-f_\ast \ge \tfrac12 \| g_1 \|^2$, and adding it to the previsously obtained inequality leads to  
    \begin{equation*}
        f_0 - f_\ast \geq \min\left(h + \tfrac12, \tfrac{1}{2} (1 - h)^{-2}\right) \|g_1\|^2,
    \end{equation*}
which is equivalent to the second bound we wanted to establish.    
\end{proof}

This result highlights the well-structured behavior of exact gradient descent, where two distinct regimes naturally emerge depending on the stepsize $h$. Furthermore, this worst-case convergence rate is \emph{tight}, meaning that it is achieved by an explicit function, and therefore cannot be improved; see \cite{rotaru2024exact} for more details and results in the exact case.\\

\subsection{Convergence rate in the inexact case ($\delta \in (0,1)$)}\label{sec::inexact_one_step}

Building on the intuition and structure of the proof in the exact case, we now turn to the relatively inexact setting. Before presenting the result, we introduce a key quantity that will play a central role in the convergence rate and its proof.

\begin{definition}\label{def::tilde_lambda}
    The quantity $\tilde{\lambda}$ is defined as the largest real root of the following cubic equation in $\lambda$, depending on both $h$ and $\delta$
    \begin{equation*}
        \left[(\delta^2-1)h^2+2h\right]\lambda^3+\left[2(\delta^2-1)h^2+5h-4\right]\lambda^2+\left[(\delta^2-1)h^2+4h-4\right]\lambda+\left[h-1\right] = 0.
    \end{equation*}
\end{definition}
\noindent This quantity, which acts as a multiplier in the proof of the next theorem, is illustrated in Figure \ref{fig::lambda}.
We now state the main result for a single iteration of relatively inexact gradient descent.

\begin{theorem}[Convergence rate of one step of inexact gradient descent on smooth convex functions]\label{theorem::full_th_inexact_one_step}
    One step of Algorithm \ref{algo::inexact_gradient} applied to a convex $L$-smooth function $f$ with an inexact gradient with relative inexactness $\delta\in(0,1)$, started from iterate $x_0$ with a constant stepsize $h\in[0,\frac{2}{1+\delta}]$,  generates an iterate $x_1$ satisfying
    \[
        \frac{1}{L} \|\nabla f(x_1)\|^2 \leq C(h, \delta) (f(x_0) - f(x_1)),
    \]
    where $C(h, \delta)$ is given by the continuous function
\small
\[
C(h, \delta) =
\left\{
\begin{array}{lll}
    \frac{1}{h(1-\delta)} & \text{if } h \in \left[0, \frac{3}{2(1+\delta)}\right) & \text{(left regime)} \\
    \frac{2\tilde{\lambda}}{h\tilde{\lambda}^2+2(h-1)\tilde{\lambda}+h-1} 
        & \text{if } h \in \left[\frac{3}{2(1+\delta)}, \frac{3\delta+2-\sqrt{4-3\delta^2}}{2\delta(\delta+1)}\right] 
        & \text{(intermediate regime)} \\
    \dfrac{2}{(1-h(1+\delta))^{-2}-1} 
        & \text{if } h \in \left(\frac{3\delta+2-\sqrt{4-3\delta^2}}{2\delta(\delta+1)}, \frac{2}{1+\delta}\right] 
        & \text{(right regime)}
\end{array}
\right.
\]
\normalsize
    and \begin{equation*}
    \frac{1}{L} \|\nabla f(x_1)\|^2 \leq \tilde{C}(h, \delta) (f(x_0) - f(x_\ast)),
    \end{equation*}
    where $\tilde{C}(h, \delta)$ is given by the continuous function
    \small
\[
\tilde{C}(h, \delta) =
\left\{
\begin{array}{lll}
    \frac{1}{h(1-\delta)+\tfrac{1}{2}}, & \text{if } h \in \left[0, \frac{3}{2(1+\delta)}\right) & \text{(left regime)} \\
    \frac{2\tilde{\lambda}}{h\tilde{\lambda}^2+(2h-1)\tilde{\lambda}+h-1}, & \text{if } h \in \left[\frac{3}{2(1+\delta)}, \frac{3\delta+2-\sqrt{4-3\delta^2}}{2\delta(\delta+1)}\right]
        & \text{(intermediate regime)} \\
    2(1-h(1+\delta))^{2}, & \text{if } h \in \left(\frac{3\delta+2-\sqrt{4-3\delta^2}}{2\delta(\delta+1)}, \frac{2}{1+\delta}\right]
        & \text{(right regime)}
\end{array}
\right.
\]
\normalsize
\end{theorem}

\begin{proof}

Recall that $d_0$ denotes the inexact gradient used for the first iteration.  We assume again, without loss of generality, that the smoothness constant is \(L = 1\). The standard homogeneity argument valid for exact gradient descent also applies to relatively inexact gradient descent, i.e., when the computed direction $d_0$ satisfies $\|d_0-\nabla f(x_0)\|\leq\delta \|\nabla f(x_0)\|$, since this condition is preserved under a scaling of function $f$. We define
\begin{equation*}
    \td \triangleq \frac{d_0 - g_o}{\delta},
\end{equation*}
so that the inequality defining the relatively inexact gradient \eqref{ine} becomes
\begin{equation*}
    \|\td\|^2 \leq \|g_0\|^2,
\end{equation*}
We write interpolation inequalities $Q_{01}$ and $Q_{10}$, now involving the inexact search direction, and the above condition on inexactness as
    \begin{equation*}
    \begin{aligned}
    Q_{01} :
         f_0 - f_1 &\geq h g_1 d_0 + \frac{g_1^2}{2} + \frac{g_0^2}{2} - g_0g_1 \\
        &= h\delta g_1 \td  + (h-1)g_0g_1 + \frac{g_1^2}{2} + \frac{g_0^2}{2}\\
        Q_{10} : 
            f_1 - f_0 &\geq -h g_0d_0 + \frac{g_1^2}{2} + \frac{g_0^2}{2} - g_0g_1 \\
            &= -h\delta g_0 \td + \left(\frac{1}{2}-h\right) g_0^2 + \frac{g_1^2}{2} - g_0g_1 \\
        Q_{inexact} : 0 &\geq \td^2 - g_0^2
    \end{aligned}
    \end{equation*}
We choose the following non-negative multipliers, where $b$ is a nonnegative parameter
        \begin{equation*}
        \begin{aligned}
            \lambda_{01} &= \lambda + 1 \\
            \lambda_{10} &= \lambda \\
            \lambda_{inexact} &= b
        \end{aligned}
        \end{equation*}
The nonnegative combination $(\lambda+1)Q_{01}+\lambda Q_{10}+bQ_{\text{inexact}}$ yields
    \begin{equation*}
        f_0-f_1\geq \frac{1}{2}
        \begin{pmatrix}
            g_0 \\
            g_1 \\
            \td
        \end{pmatrix}
        \left[
        A
        +
        \begin{pmatrix}
            0&0&0\\
            0&2\rho&0\\
            0&0&0
        \end{pmatrix}\right]
        \begin{pmatrix}
            g_0 &
            g_1 &
            \td
        \end{pmatrix},
    \end{equation*}
with
\begin{equation*}
    A=\begin{pmatrix}
            2(1-h)\lambda-2b+1 & (h-2)\lambda-1+h & -h\delta\lambda \\
            (h-2)\lambda -1+h & 2\lambda+1-2\rho & h\delta(\lambda+1) \\
            -h\delta\lambda & h\delta(\lambda+1)& 2b
        \end{pmatrix}
\end{equation*}
where, similarly to the proof of the exact case, we introduce $\rho$ to denote the coefficient of $\|g_1\|^2$ in the rate. We decompose the matrix $A$ using the Schur complement \cite{zhang2006schur} with respect to the bottom right element. This leads to the sum $A=A_1+A_2$, where

\begin{equation*}
\begin{aligned}
    A_1&=\begin{pmatrix}
    - 2 b - 2 \lambda \left(h - 1\right) + 1 - \frac{\delta^{2} h^{2} \lambda^{2}}{2b} & h + \lambda \left(h - 2\right) - 1 + \frac{\delta^{2} h^{2} \lambda \left(\lambda + 1\right)}{2b}&0\\h + \lambda \left(h - 2\right) - 1 + \frac{\delta^{2} h^{2} \lambda \left(\lambda + 1\right)}{2b} & - 2 \rho + 2 \lambda + 1 - \frac{\delta^{2} h^{2} \left(\lambda + 1\right)^{2}}{2 b}&0\\0&0&0
    \end{pmatrix},
    \\
    A_2&=\begin{pmatrix}
    \frac{\delta^{2} h^{2} \lambda^{2}}{2 b} & - \frac{\delta^{2} h^{2} \lambda \left(\lambda + 1\right)}{2 b}&-h\delta\lambda\\- \frac{\delta^{2} h^{2} \lambda \left(\lambda + 1\right)}{2 b} & \frac{\delta^{2} h^{2} \left(\lambda + 1\right)^{2}}{2 b}&h\delta(\lambda+1)\\-h\delta\lambda&h\delta(\lambda+1)&2b
    \end{pmatrix}.
    \end{aligned}
\end{equation*}
Matrix $A_2$ is rank-1 and positive semidefinite matrix and can be written as
\begin{equation*}A_2=
    \begin{pmatrix}
        -\frac{h\delta\lambda}{\sqrt{2b}} \\ \frac{h\delta(\lambda+1)}{\sqrt{2b}}\\\sqrt{
        2b}
    \end{pmatrix}
    \begin{pmatrix}
        -\frac{h\delta\lambda}{\sqrt{2b}} &\frac{h\delta(\lambda+1)}{\sqrt{2b}}&\sqrt{
        2b}
    \end{pmatrix},
\end{equation*}
Hence, matrix $A$ will be positive semidefinite as soon as this is the case for matrix $A_1$, and maximizing the worst-case bound corresponds to maximizing $\rho$ under the constraint that $A_1\succeq0$. It is easy to show that $A_1\succeq0$ is equivalent to one of the following
\begin{itemize}
    \item $A_{1_{0,0}}>0$ and $\det(A_1)\geq0$ (Cases $1$ and $2$ below).
    \item $A_{1_{0,0}}=0$, $A_{1_{1,0}}=A_{1_{0,1}}=0$, and $A_{1_{1,1}}\geq0$ (Case 3 below).
\end{itemize}
The first two cases lead to straightforward but lengthy computations, which can be carried out symbolically using the Python library \texttt{SymPy} \cite{10.7717/peerj-cs.103}. The corresponding code is available at: \url{https://github.com/Verpierre/Inexact_gradient_descent}. Some further symbolic simplifications were performed with the software Wolfram Mathematica \cite{Mathematica}.

We now focus on the first two cases, for which the feasibility domain of $b$ and the corresponding upper bound on $\rho$ read

\scriptsize
\begin{equation}\label{eq::main_problem}
    \begin{aligned}
        \max_{\lambda,b,\rho}\quad &\rho\\
        \text{s.t.}\quad  b &> \frac{1}{4}+\frac{\lambda}{2} \left(1-h\right) - \frac{\sqrt{\left(2\lambda (h(1-\delta)-1) - 1\right) \left(2\lambda (h(1+\delta)-1) - 1\right)}}{4}\quad(c_{1b})\\
        b&< \frac{1}{4}+\frac{\lambda}{2} \left(1-h\right) + \frac{\sqrt{\left(2\lambda (h(1-\delta)-1) - 1\right) \left(2\lambda (h(1+\delta)-1) - 1\right)}}{4}\quad(c_{1b})\\
        \rho&\leq \frac{b^2(8 \lambda +4)+b\left(-2 h \left(\left(\delta
   ^2-1\right) h (\lambda +1)^2+4 \lambda +2\right)\right)+\delta
   ^2 h^2 (2 \lambda +1)}{8 b^2+b(8 (h-1) \lambda -4)+2 \delta ^2
   h^2 \lambda ^2}\quad(c_2)
    \end{aligned}
\end{equation}
\normalsize

Considering the natural extension of the solutions of the exact case (see the proof of Theorem \ref{th::exact_convergence_rate}) leads to the following cases.

\begin{description}
   \item[Case 1: $\lambda=1$.]

   The problem becomes

   \begin{equation*}
    \begin{aligned}
        \max_{b,\rho}\quad &\rho\\
        \text{s.t.}\quad  b &> \frac{3}{4}- \frac{h}{2} - \frac{\sqrt{4(1-\delta^2) h^{2} - 12 h + 9}}{4}\quad(c_{11a}) \\
        b &<\frac{3}{4}- \frac{h}{2} + \frac{\sqrt{4(1-\delta^2)h^2 - 12 h + 9}}{4} \quad(c_{11b})\\
        \rho&\leq \frac{12b^2+b\left(-2 h \left(4 \left(\delta ^2-1\right) h+6\right)\right)+3 \delta ^2 h^2}{8 b^2+b(8 h-12)+2 \delta ^2h^2}\quad(c_{12})
    \end{aligned}
\end{equation*}
We denote the objective function as $\rho(b)$. As it is maximized, the problem is equivalent to
\begin{equation*}
    \begin{aligned}
        \max_{b}\quad &\rho(b)=\frac{12b^2+b\left(-2 h \left(4 \left(\delta ^2-1\right) h+6\right)\right)+3 \delta ^2 h^2}{8 b^2+b(8 h-12)+2 \delta ^2h^2}\\
        \text{s.t.}\quad  b &> \frac{3}{4}- \frac{h}{2} - \frac{\sqrt{4(1-\delta^2) h^{2} - 12 h + 9}}{4}\quad(c_{11a}) \\
        b &<\frac{3}{4}- \frac{h}{2} + \frac{\sqrt{4(1-\delta^2)h^2 - 12 h + 9}}{4} \quad(c_{11b})\\
    \end{aligned}
\end{equation*}

The first-order optimality condition for the objective function is
\begin{equation*}
    \frac{d\rho(b)}{db}=\frac{\left(2 b - \delta h\right) \left(2 b + \delta h\right) \left(2 \delta h - 2 h + 3\right) \left(2 \delta h + 2 h - 3\right)}{\left(4 b^{2} + 4 b h - 6 b + \delta^{2} h^{2}\right)^{2}}=0
\end{equation*}
whose maximum is reached when $b_\ast=\frac{h\delta}{2}$. We now check for which values of $h$ this value lies inside the admissible domain from the constraints $c_{11a}$ and $c_{11b}$.
Solving both irrational inequations in $(c_{11})$ leads to the following condition on $h$ for $b_\ast$ to be admissible
\begin{equation*}
    h < \frac{3}{2(1 + \delta)}.
\end{equation*}
Thus, the solution $b_\ast=\frac{h\delta}{2}$ holds as soon as $h< \frac{3}{2(1+\delta)}$. Furthermore, replacing $b$ by its optimal value in the expression of the rate $\rho(b)$ leads to
\begin{equation*}
    \begin{aligned}
        \lambda&=1\\
        b&=\frac{h\delta}{2}\\
        \rho&=h(1-\delta),
    \end{aligned}
\end{equation*}
when $h<\frac{3}{2(1+\delta)}$. Note that extending the multiplier $\lambda$ from the exact case leads to a rate that is very close to the exact one: it is scaled by a factor $1-\delta$, and equal to the exact case when $\delta=0$).

\item[Case 2: $\lambda=\frac{2-h(1+\delta)}{h(1+\delta)-1}$.]

This second case is an extension of the exact one (see the proof of Theorem \ref{th::exact_convergence_rate}), where $h\leftarrow h(1+\delta)$. In this case, problem \eqref{eq::main_problem} can be rewritten as

\tiny
\begin{equation*}
    \begin{aligned}
        \max_{b,\rho}\quad &\rho\\
        \text{s.t.}\quad  b &> \frac{1}{4}+\frac{\left(1- h\right) \left(- h \left(\delta + 1\right) + 2\right)}{2(h \left(\delta + 1\right) - 1)}-\frac{1}{4} \sqrt{\left(- \frac{\left(- h \left(\delta + 1\right) + 2\right) \left(2 \delta h - 2 h + 2\right)}{h \left(\delta + 1\right) - 1} - 1\right) \left(\frac{\left(- h \left(\delta + 1\right) + 2\right) \left(2 \delta h + 2 h - 2\right)}{h \left(\delta + 1\right) - 1} - 1\right)} \quad(c_{21a})\\
        b &< \frac{1}{4}+\frac{\left(1- h\right) \left(- h \left(\delta + 1\right) + 2\right)}{2(h \left(\delta + 1\right) - 1)}+\frac{1}{4} \sqrt{\left(- \frac{\left(- h \left(\delta + 1\right) + 2\right) \left(2 \delta h - 2 h + 2\right)}{h \left(\delta + 1\right) - 1} - 1\right) \left(\frac{\left(- h \left(\delta + 1\right) + 2\right) \left(2 \delta h + 2 h - 2\right)}{h \left(\delta + 1\right) - 1} - 1\right)} \quad(c_{21b})\\
        \rho&\leq -\frac{b^2(4 (\delta  h+h-3) (\delta  h+h-1))+b(-2 h ((\delta +1)h (-\delta +2 (\delta +1) h-7)+6))+\delta ^2 h^2 (\delta h+h-3) (\delta  h+h-1)}{2 \left(b^2\left(4 (\delta h+h-1)^2\right)+b\left(-2 (\delta  h+h-1) \left(2 (\delta +1)h^2-(\delta +5) h+3\right)\right)+\delta ^2 h^2 (\delta h+h-2)^2\right)}
    \end{aligned}
\end{equation*}
\normalsize
which is equivalent to
\tiny
\begin{equation*}
    \begin{aligned}
        \max_{b}\quad &\rho(b)=-\frac{b^2(4 (\delta  h+h-3) (\delta  h+h-1))+b(-2 h ((\delta +1)h (-\delta +2 (\delta +1) h-7)+6))+\delta ^2 h^2 (\delta h+h-3) (\delta  h+h-1)}{2 \left(b^2\left(4 (\delta h+h-1)^2\right)+b\left(-2 (\delta  h+h-1) \left(2 (\delta +1)h^2-(\delta +5) h+3\right)\right)+\delta ^2 h^2 (\delta h+h-2)^2\right)}\\
        \text{s.t.}\quad  &b\geq \frac{1}{4}+\frac{\left(1- h\right) \left(- h \left(\delta + 1\right) + 2\right)}{2(h \left(\delta + 1\right) - 1)}-\frac{1}{4} \sqrt{\left(- \frac{\left(- h \left(\delta + 1\right) + 2\right) \left(2 \delta h - 2 h + 2\right)}{h \left(\delta + 1\right) - 1} - 1\right) \left(\frac{\left(- h \left(\delta + 1\right) + 2\right) \left(2 \delta h + 2 h - 2\right)}{h \left(\delta + 1\right) - 1} - 1\right)} \\
        &b \leq \frac{1}{4}+\frac{\left(1- h\right) \left(- h \left(\delta + 1\right) + 2\right)}{2(h \left(\delta + 1\right) - 1)}+\frac{1}{4} \sqrt{\left(- \frac{\left(- h \left(\delta + 1\right) + 2\right) \left(2 \delta h - 2 h + 2\right)}{h \left(\delta + 1\right) - 1} - 1\right) \left(\frac{\left(- h \left(\delta + 1\right) + 2\right) \left(2 \delta h + 2 h - 2\right)}{h \left(\delta + 1\right) - 1} - 1\right)}
    \end{aligned}.
\end{equation*}
\normalsize

The first-order optimality condition on the objective function yields
\tiny
\begin{equation*}
    \frac{d\rho(b)}{db}=\frac{(2 (\delta +1) h-3) (2 b (\delta  h+h-1)-\delta  h) (\delta h (2 h (\delta  (\delta  h+h-3)-1)+3)-2 b (2 h-3) (\delta h+h-1))}{\left(b^2\left(4 (\delta  h+h-1)^2\right)+b\left(-2(\delta  h+h-1) \left(2 (\delta +1) h^2-(\delta +5)h+3\right)\right)+\delta ^2 h^2 (\delta  h+h-2)^2\right)^2}=0
\end{equation*}
\normalsize
and this rate is maximized for $b_\ast=\frac{h\delta}{2(h(1+\delta)-1)}$. Constraints  $(c_{21a})$ and $(c_{21b})$ are satisfied as soon as $h > \frac{3\delta+2-\sqrt{4-3\delta^2}}{2\delta(\delta+1)}$.
Replacing $b$ with its optimal value leads to
\begin{equation*}
    \begin{aligned}
        \lambda &=\frac{2-h(1+\delta)}{h(1+\delta)-1}\\
        b&=\frac{h\delta}{2(h(1+\delta)-1)}\\
        \rho&=\frac{h(1+\delta)(2-h(1+\delta))}{2(h(1+\delta)-1)^2},
    \end{aligned}
\end{equation*}
for $h > \frac{3\delta+2-\sqrt{4-3\delta^2}}{2\delta(\delta+1)}$. In this case also, we can see a natural evolution of the rate from the exact case, and we can, in particular, recover it when $\delta=0$.

\item[Case 3: Tight constraint on $b$.]

We now handle the remaining regime, corresponding to stepsizes ($h\in[\frac{3}{2(1+\delta)},\frac{3\delta+2-\sqrt{4-3\delta^2}}{2\delta(1+\delta)}]$), where the constraint $A_{1_{0,0}}=0$ is saturated. This leads to the following equality involving the variables $b$, $\lambda$, and $h$
\begin{equation}\label{inter_b_l}
\begin{aligned}
    - 2 b^{2} +\left(2\lambda(1-h)+1\right)b - \frac{\delta^{2} h^{2} \lambda^{2}}{2}&= 0\\
    b&\geq0
\end{aligned}
\end{equation}
Under this condition, matrix $A_1$ reduces to the form

\begin{equation*}
A_1 = \begin{pmatrix}
0 & A_{1_{0,1}} & 0 \\
A_{1_{1,0}} & A_{1_{1,1}} & 0 \\
0 & 0 & 0
\end{pmatrix},
\end{equation*}
where
\begin{align*}
A_{1_{1,0}} &= A_{1_{0,1}} = h + \lambda(h - 2) - 1 + \frac{\delta^2 h^2 \lambda(\lambda + 1)}{2b}, \\
A_{1_{1,1}} &= -2\rho + 2\lambda + 1 - \frac{\delta^2 h^2 (\lambda + 1)^2}{2b}.
\end{align*}

The matrix is positive semidefinite if and only if  $A_{1_{1,0}}=A_{1_{0,1}}=0$ and $A_{1_{1,1}}\geq0$. These conditions yield
\begin{equation}\label{inter_rate_l_b}
    \begin{aligned}
        &\left(h + \lambda \left(h - 2\right) - 1\right) + \frac{\delta^{2} h^{2} \lambda \left(\lambda + 1\right)}{2b}=0\\
        &\rho\leq\lambda+0.5-\frac{\delta^2h^2(\lambda+1)^2}{4b}
    \end{aligned}
\end{equation}

The first equation in \eqref{inter_rate_l_b} can be rewritten as
\begin{equation*}
    -\frac{\delta^2h^2\lambda(\lambda+1)}{4b}=\frac{1}{2}(h+\lambda(h-2)-1)
\end{equation*}
We can now substitute this expression into the inequality on $C$, which is tight at optimality (i.e., the inequality becomes an equality). This gives
\begin{equation*}
\begin{aligned}
    \rho &= \lambda+0.5-\frac{\delta^2h^2(\lambda+1)^2}{4b}\\
    &=\lambda+0.5-\frac{\delta^2h^2}{4b}\lambda(\lambda+1)-\frac{\delta^2h^2}{4b}\frac{\lambda(\lambda+1)}{\lambda}\\
    &= \lambda+0.5-\frac{1}{2}\left(h-1+\lambda(h-2)\right)\left(1+\frac{1}{\lambda}\right)\\
    &=h-1+\frac{h}{2}\lambda+\frac{h-1}{2}\frac{1}{\lambda}
\end{aligned}
\end{equation*}

Now, using the first equation of \eqref{inter_rate_l_b} in the quadratic relation \eqref{inter_b_l}, we get
\begin{equation*}
    \begin{aligned}
        -2b^2+(2\lambda(1-h)+1)b&=\frac{\delta^2h^2}{2}\lambda^2\\
        &=-b\left[(h-2)\lambda+1-\frac{1}{\lambda+1}\right],
    \end{aligned}
\end{equation*}
hence, since $b\neq0$, we find
\begin{equation*}
\begin{aligned}
    2b&=(2\lambda(1-h)+1)+(h-2)\lambda+1-\frac{1}{\lambda+1}\\
    &=2-h\lambda-\frac{1}{\lambda+1}
\end{aligned}
\end{equation*}

Thus, we obtain closed-form expressions for $b$ and $C$ as functions of $\lambda$
\begin{equation*}
    \begin{aligned}
        \rho(\lambda)&=h-1+\frac{h}{2}\lambda+\frac{h-1}{2}\frac{1}{\lambda}\\
        b(\lambda)&=1-\frac{h}{2}\lambda-\frac{1}{2(1+\lambda)}
    \end{aligned}
\end{equation*}

In addition, the constraint from \eqref{inter_rate_l_b} can now be written entirely in terms of $\lambda$
\begin{equation*}
    0=(h-1+\lambda(h-2))+\frac{\delta^2h^2\lambda(\lambda+1)}{2b(\lambda)}
\end{equation*}

Substituting the expression of $b(\lambda)$ into this equation yields a cubic equation in $\lambda$, which defines $\tilde{\lambda}$ as in Definition \ref{def::tilde_lambda}. The convergence rate is maximized for $\lambda_\ast=\tilde{\lambda}$, the largest real root of this equation. Ultimately, we obtain the following
\begin{equation*}
\begin{aligned}
    &\tilde{\lambda} \text{ as defined in Definition \ref{def::tilde_lambda}} \\
    &\rho=h-1+\frac{h}{2} \tilde{\lambda}+\frac{h-1}{2}\frac{1}{ \tilde{\lambda}}=\frac{h\tilde{\lambda}^2+2(h-1)\tilde{\lambda}+h-1}{2\tilde{\lambda}}\\
    &b=1-\frac{h}{2} \tilde{\lambda}-\frac{1}{2(1+ \tilde{\lambda})}
\end{aligned}
\end{equation*}
\end{description}
To conclude, we take $C(h,\delta)=\frac{1}{\rho}$ and discussing the value of $\lambda_\ast$ as a function of $h$ and $\delta$, which leads to the first bound in the Theorem, involving $f(x_0)-f(x_1)$ in the right-hand side. The second bound with $f(x_0)-f(x_\ast)$ is then obtained with the argument used at the end of the proof of Theorem \ref{th::exact_convergence_rate}.

\end{proof}

This convergence rate, depending on the stepsize $h$ according to three distinct regimes, can be visualized in Figure \ref{fig:convergence_rate_inexact} for different values of $\delta$, while the evolution of the multiplier $\lambda$ with $h$ is shown in Figure \ref{fig::lambda}.

\begin{figure}[H]
    \centering
    \includegraphics[width=0.5\linewidth]{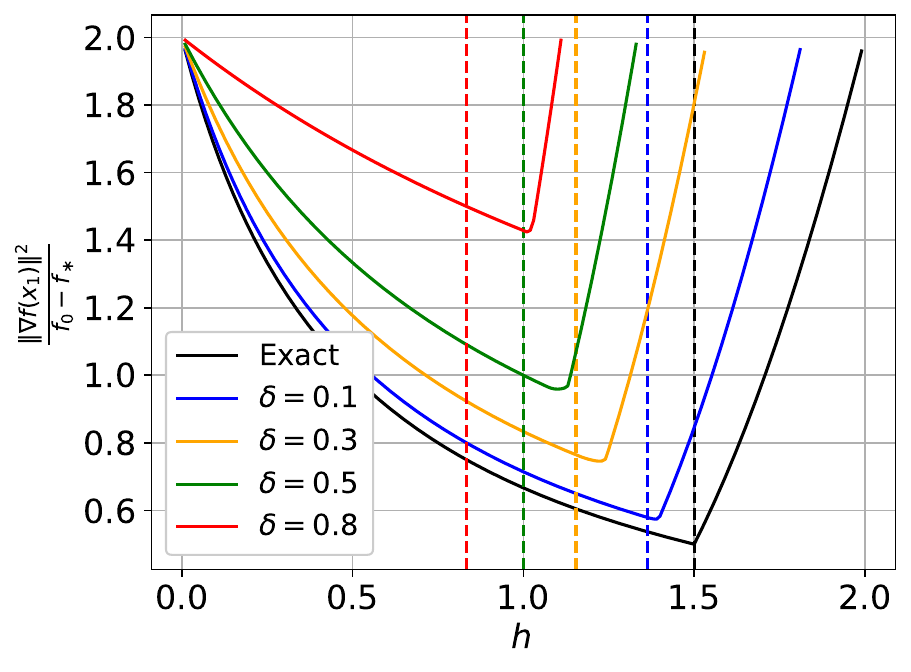}
    \caption{\centering Convergence rate of one step of exact and inexact gradient descents with several levels of inexactness $\delta$, depending on the stepsize $h$. Vertical dotted lines mark the particular stepsize $h=\frac{3}{2(1+\delta)}$}
    \label{fig:convergence_rate_inexact}
\end{figure}

\begin{figure}[H]
    \centering
    \begin{subfigure}{0.32\textwidth}
        \includegraphics[width=\textwidth]{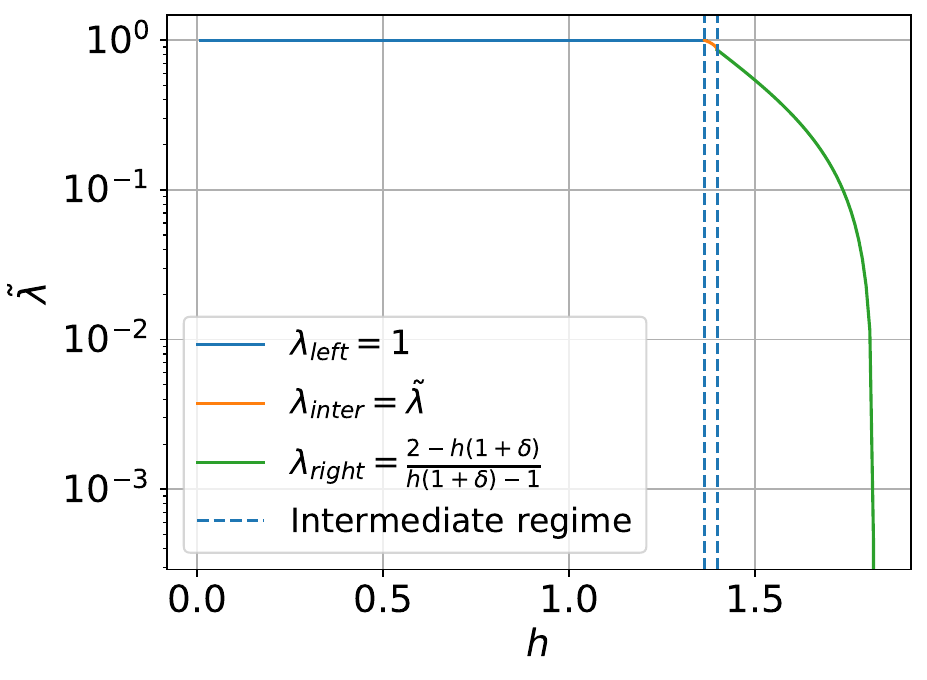}
        \caption{$\delta = 0.1$}
    \end{subfigure}
    \begin{subfigure}{0.32\textwidth}
        \includegraphics[width=\textwidth]{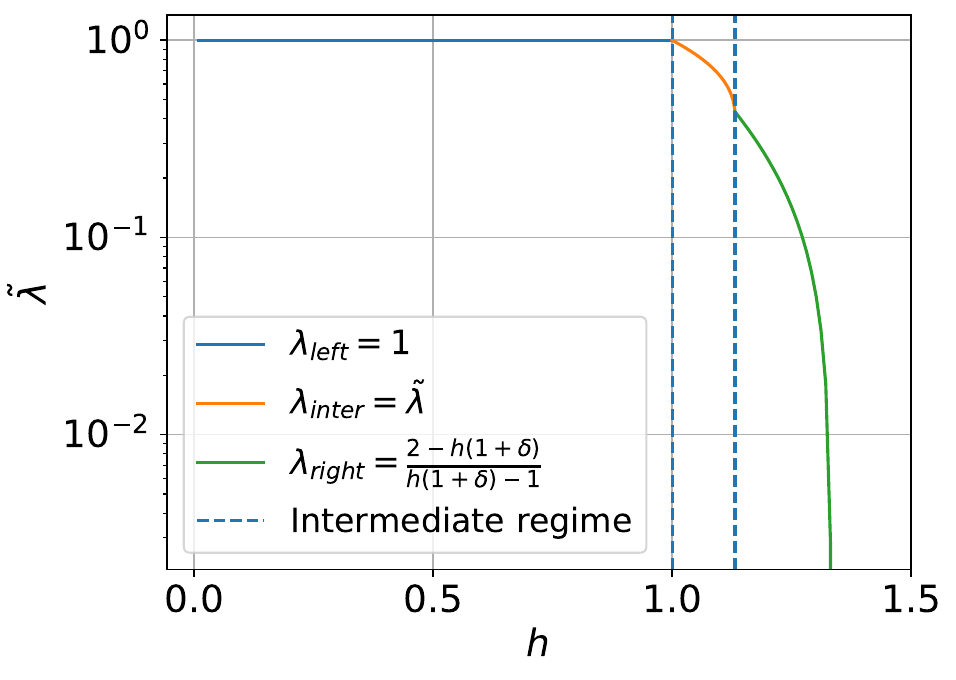}
        \caption{$\delta = 0.5$}
    \end{subfigure}
    \begin{subfigure}{0.32\textwidth}
        \includegraphics[width=\textwidth]{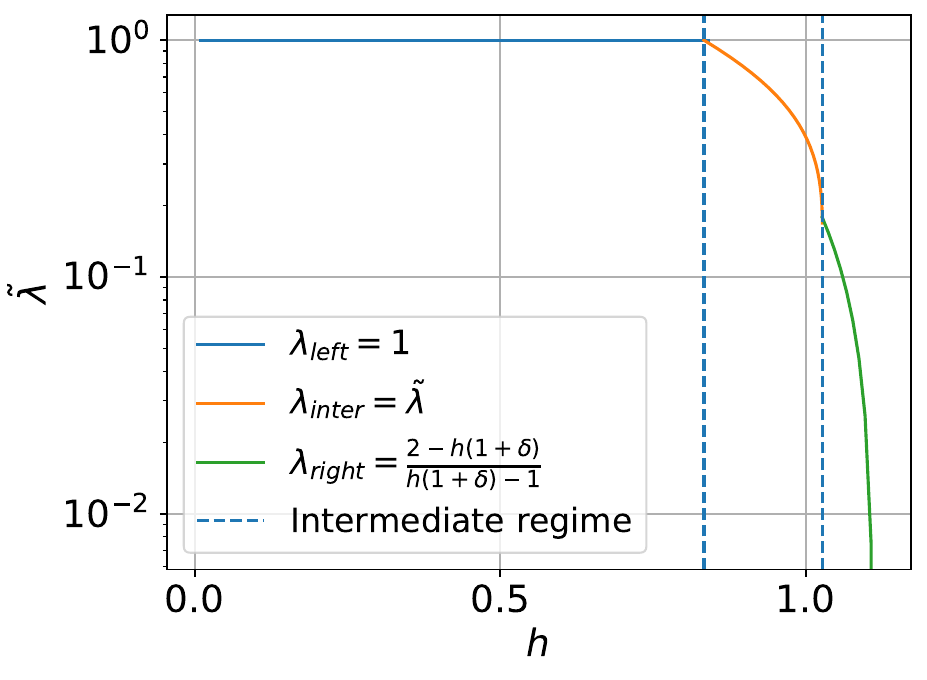}
        \caption{$\delta = 0.8$}
    \end{subfigure}
    \caption{\centering Multiplier $\lambda(h)$ from the proof of Theorem \ref{theorem::full_th_inexact_one_step} for different values of $\delta$ along the three regimes. This shows that the proof is \emph{continuous} over all step sizes $\in [0,\frac{2}{1+\delta}]$.}
    \label{fig::lambda}
\end{figure}

Compared to the exact case, where two stepsize regimes naturally arise (short-step and long-step), an interesting phenomenon occurs in the inexact setting: the emergence of a third distinct regime. We refer to the first regime, $h \in \left[0,\frac{3}{2(1+\delta)}\right)$, as the left or short-step regime. It is followed by the intermediate regime, $h \in \left[\frac{3}{2(1+\delta)}, \frac{3\delta+2-\sqrt{4-3\delta^2}}{2\delta(\delta+1)}\right]$, and finally the right or long-step regime for $h \in [\frac{3\delta+2-\sqrt{4-3\delta^2}}{2\delta(\delta+1)}, \frac{2}{1+\delta}]$.

The short-step and long-step regimes can be seen as natural extensions of the exact case. However, the intermediate regime is entirely new and arises uniquely from the presence of inexactness in the gradient. It does not correspond to any continuous extension of the behavior seen in the exact setting. Instead, it appears as a separate solution branch resulting from the structure of the positivity constraints in the semidefinite programming formulation. Interestingly, numerical experiments show that a similar tripartite regime structure also appears in settings with absolute inexactness, though a thorough analysis of that case falls outside the scope of this paper.

An important point is that this intermediate regime, while narrow, contains the optimal stepsize, i.e., the one that minimizes the worst-case convergence rate. However, characterizing this optimal value analytically is challenging: it results from solving the third-order polynomial equation from Definition \ref{def::tilde_lambda} with coefficients depending on both the stepsize $h$ and the inexactness parameter $\delta$. These equations are generally not solvable in closed form, even with symbolic tools, making a precise analytical expression for the optimal rate intractable.

We observe this behavior clearly in Figure \ref{fig:rate_with_zoom}, which shows the convergence rate curves for three different values of $\delta$. The zoom-in highlights the intermediate regime, revealing a sharp minimum in the rate curve. This confirms that choosing a stepsize within the intermediate regime yields the best performance. Moreover, a stepsize chosen too large (i.e., beyond the end of the intermediate regime) leads to a rapid deterioration of the convergence rate. This illustrates a form of asymmetry: it's safer to err on the side of smaller stepsizes, as overly aggressive steps can significantly degrade performance. This behavior is similar to what is observed in the exact case \cite{taylor2017smooth, vernimmentight}.

\begin{figure}[H]
    \centering
\begin{subfigure}{0.32\textwidth}
  \includegraphics[width=\linewidth]{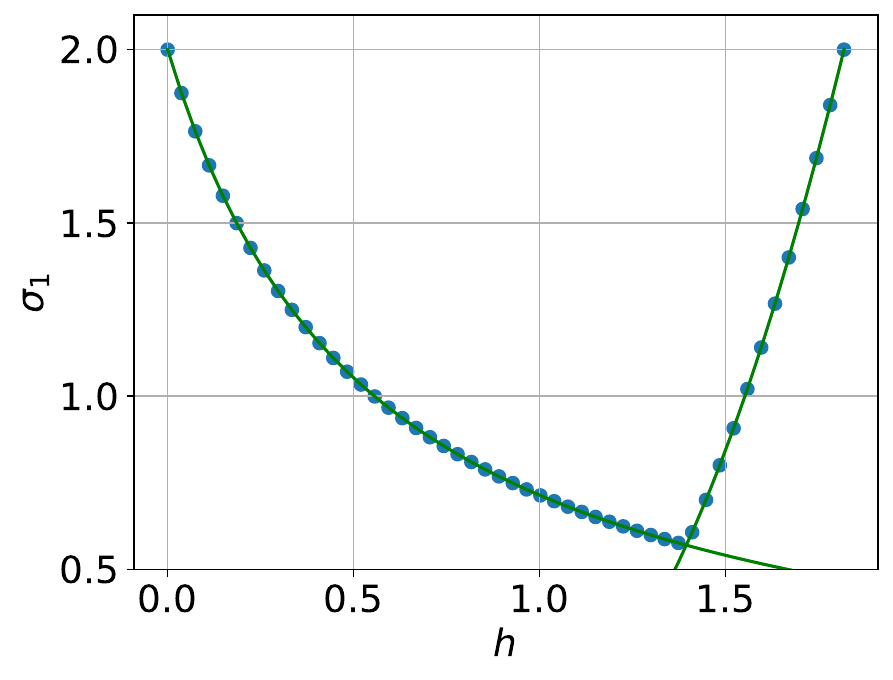}
  \caption{\centering $\delta=0.1$, $h \in \left[0, \frac{2}{1+\delta}\right]$}
\end{subfigure}\hfil %
\begin{subfigure}{0.32\textwidth}
  \includegraphics[width=\linewidth]{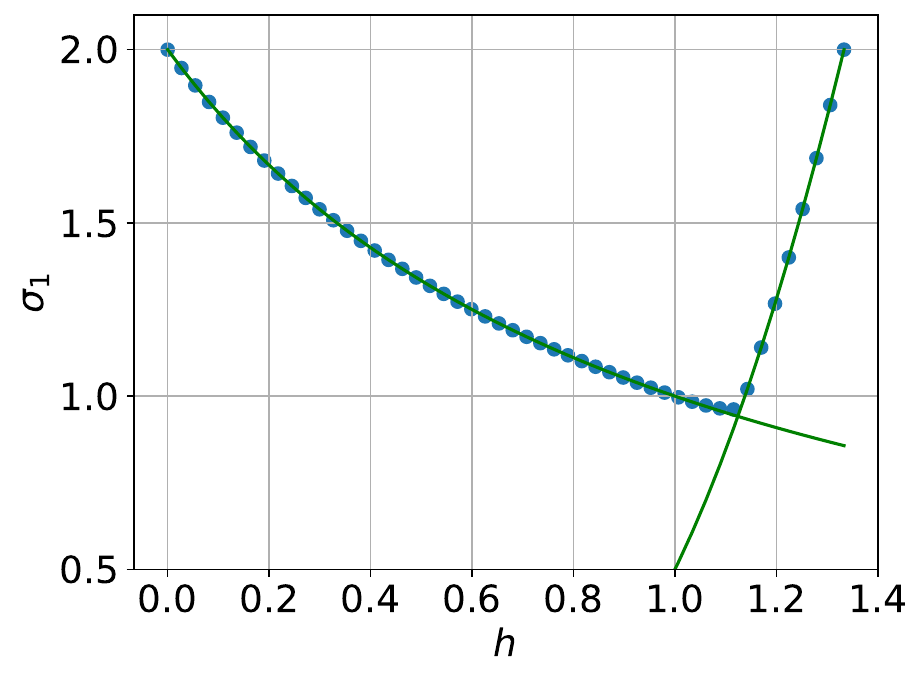}
  \caption{\centering $\delta=0.5$, $h \in \left[0, \frac{2}{1+\delta}\right]$}
\end{subfigure}\hfil 
\begin{subfigure}{0.32\textwidth}
  \includegraphics[width=\linewidth]{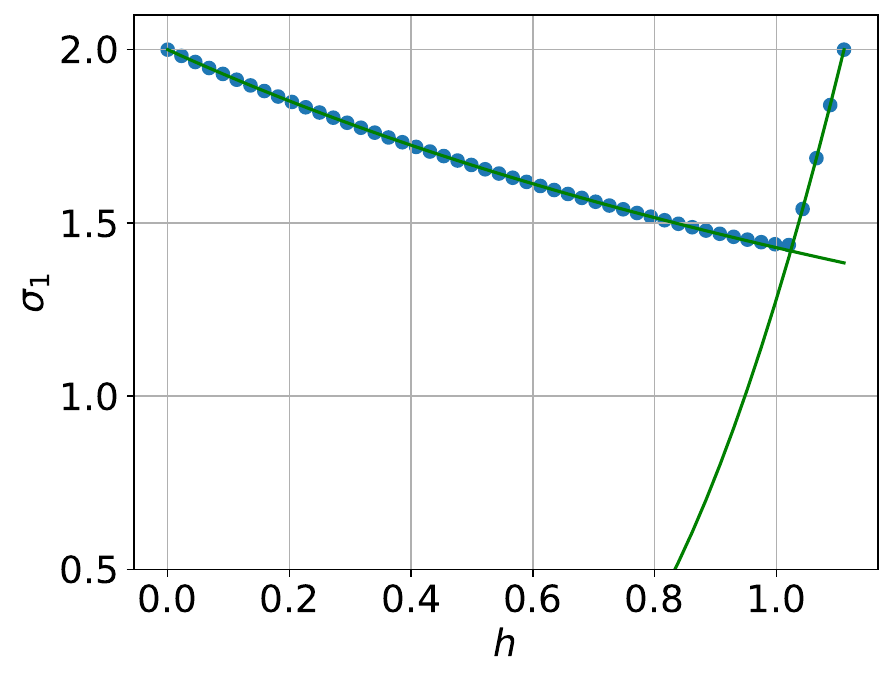}
  \caption{\centering $\delta=0.8$, $h \in \left[0, \frac{2}{1+\delta}\right]$}
\end{subfigure}

\medskip
\begin{subfigure}{0.32\textwidth}
  \includegraphics[width=\linewidth]{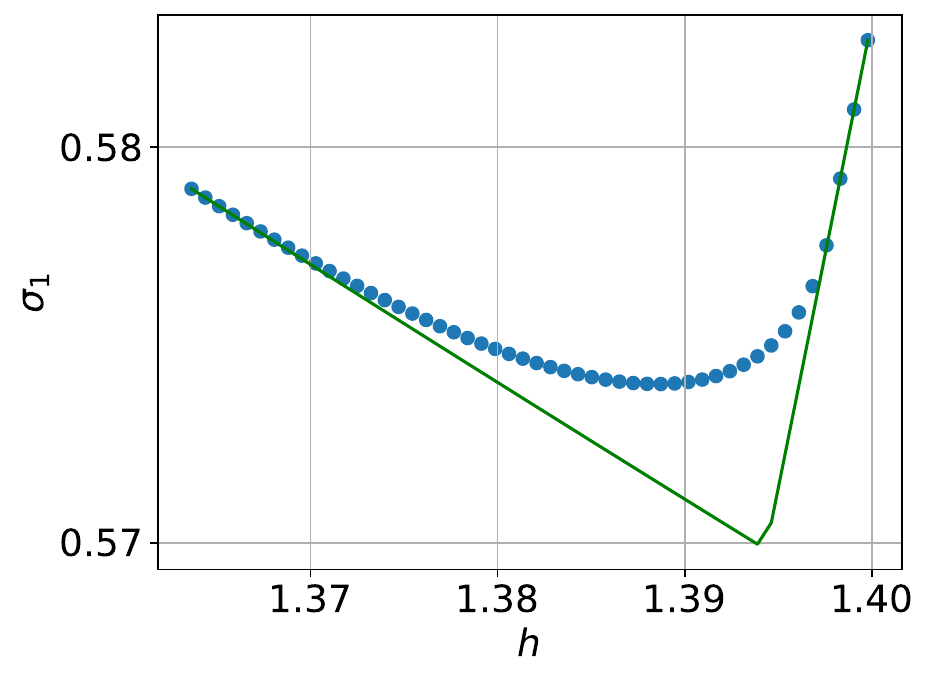}
  \caption{\centering $\delta=0.1$, $h \in \left[\frac{3}{2(1+\delta)}, \frac{3\delta+2-\sqrt{4-3\delta^2}}{2\delta(\delta+1)}\right]$}
\end{subfigure}\hfil 
\begin{subfigure}{0.32\textwidth}
  \includegraphics[width=\linewidth]{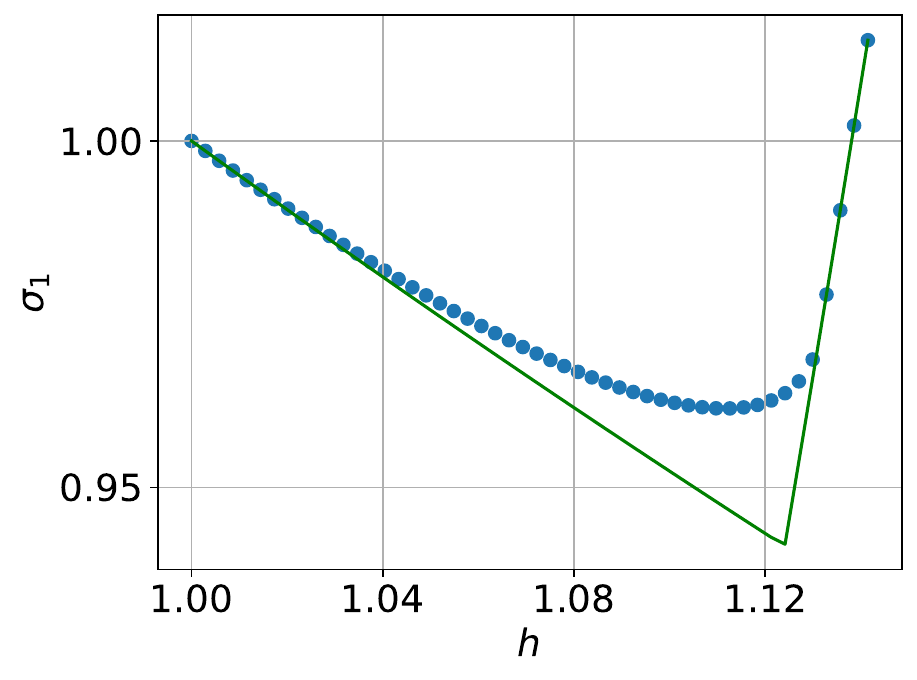}
  \caption{\centering $\delta=0.5$, $h \in \left[\frac{3}{2(1+\delta)}, \frac{3\delta+2-\sqrt{4-3\delta^2}}{2\delta(\delta+1)}\right]$}
\end{subfigure}\hfil 
\begin{subfigure}{0.32\textwidth}
  \includegraphics[width=\linewidth]{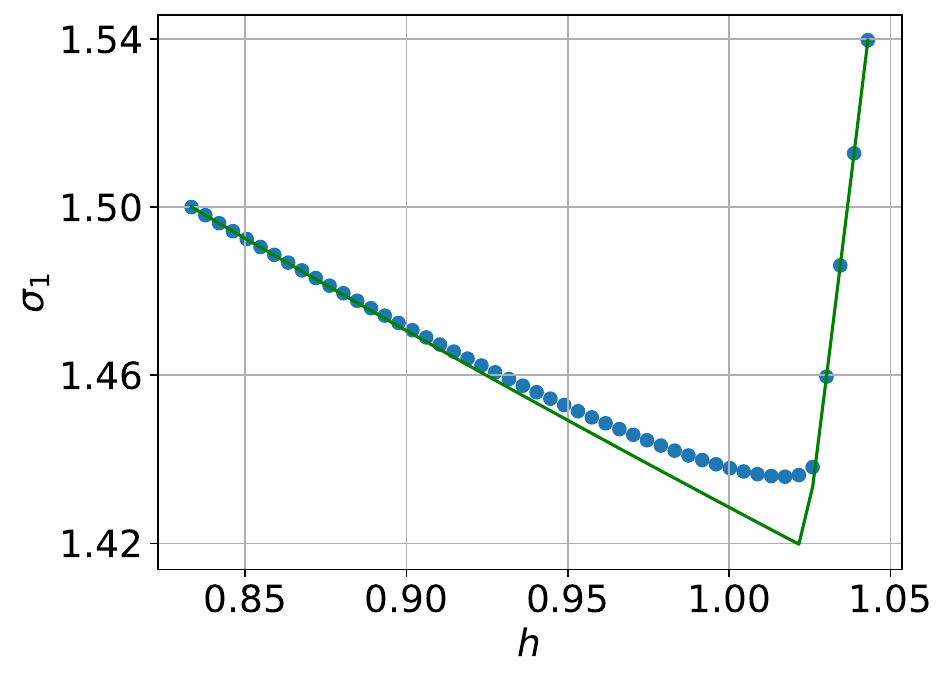}
  \caption{\centering $\delta=0.8$, $h \in \left[\frac{3}{2(1+\delta)}, \frac{3\delta+2-\sqrt{4-3\delta^2}}{2\delta(\delta+1)}\right]$}
\end{subfigure}
\caption{\centering (top) Convergence rate (blue dots) for different levels of inexactness for all stepsizes $h \in \left[0, \frac{2}{1+\delta}\right]$, and (bottom) zooming on the intermediate regime. Solid lines are the natural extension of the exact case and correspond to the other regimes.}
\label{fig:rate_with_zoom}
\end{figure}

We emphasize another key structural difference between the exact and relatively inexact settings: the region of guaranteed convergence for the normalized stepsize \( h \) shrinks as the inexactness increases. In the exact case (\( \delta = 0 \)), convergence is ensured for all \( h < 2 \). However, under relative inexactness, convergence is guaranteed only when \( h < \frac{2}{1+\delta} \). This dependence on \( \delta \) reflects a fundamental limitation introduced by the errors: the maximal admissible stepsize necessarily decreases with increasing inexactness. This qualitative behavior, specific to the relatively inexact setting, highlights the importance of adapting the stepsize to the inexactness level. In particular, the factor \( \frac{1}{1+\delta} \leq 1 \) appears repeatedly in the analysis, and has motivated an empirical strategy of shortening stepsizes by this factor in \cite{vernimmen2025empirical}, with encouraging results. A broader discussion on stepsize selection over multiple iterations is provided after Figure \ref{fig:lower_upper_several_steps} in the next section.

Quantitatively speaking, our analysis also significantly improves upon the result of \cite{vasin2024gradient} (Theorems 2.4 and 2.5), where convergence is only established for $h < h_{\text{max}}=2\left(\frac{1-\delta}{1+\delta}\right)^{3/2} $. Our maximal stepsize $h_{\text{max}}=\frac{2}{1+\delta}$ allows for substantially larger values of $h$, especially when $\delta$ is large. This improvement is illustrated in Figure \ref{fig::h_max}, where the two bounds are compared. It confirms that our theoretical bound is not only tight but also more permissive than previous results, allowing for faster convergence in practice without compromising stability.

\begin{figure}[H]
    \centering
    \includegraphics[width=0.5\linewidth]{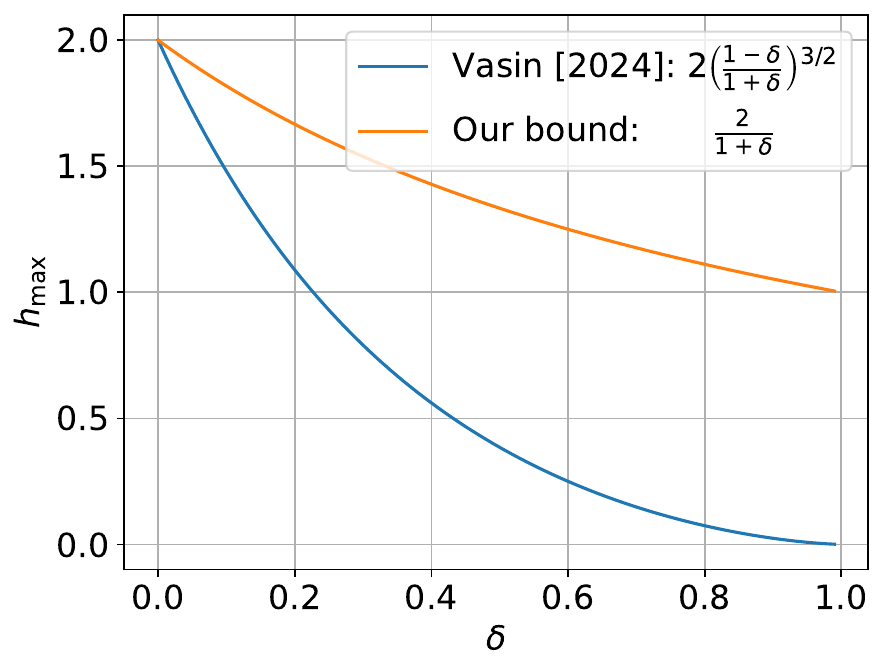}
    \caption{\centering Comparison of our $h_{\text{max}}$ for which convergence is guaranteed with the one from \cite{vasin2024gradient}.}
    \label{fig::h_max}
\end{figure}

\subsection{Tightness of the rate in the left and right regimes}

After establishing worst-case convergence rates for the inexact gradient descent in Theorem \ref{theorem::full_th_inexact_one_step}, we now show that these rates cannot be improved. For the left and right regimes, we provide an explicit function and a specific inexact direction for which the behavior of the method exactly matches the theoretical bound. In the intermediate regime, while we could not identify a similar explicit worst-case instance, we conjecture that the rate is also tight based on strong numerical evidence. 
We formalize this in the following result:

\begin{theorem}[Tightness of left and right regimes]\label{th::tight_inexact}
    The left and right rates in Theorem \ref{theorem::full_th_inexact_one_step} are tight. Moreover, they are attained by univariate worst-case functions: a Huber function for the left regime and a quadratic function for the right regime.
\end{theorem}

\begin{proof}
We prove the tightness for the second rate of Theorem \ref{theorem::full_th_inexact_one_step} (the one with $f(x_0)-f(x_\ast)$), but the first one can be proved in a similar way (and with the same function).

    \begin{description}
        \item[Left regime $h\leq \frac{3}{2(1+\delta)}$]
        We consider a univariate Huber-like function, which is $1$-smooth and convex

     \begin{equation*}
     f(x) =
    \left\{
    \begin{array}{rl}
   \tfrac{1}{\sqrt{h(1-\delta)+\tfrac{1}{2}}} |x| - \tfrac{1}{2h(1-\delta)+1} ~~\text{, when } |x| \geq \tfrac{1}{\sqrt{h(1-\delta)+\tfrac{1}{2}}} \\
    \frac{x^2}{2}~~\text{, when } |x| < \tfrac{1}{\sqrt{h(1-\delta)+\tfrac{1}{2}}}
    \end{array}
    \right. 
    \end{equation*}

    which admits a unique minimizer for $x=0$ and $f(0)=0$.
     
    \noindent Assuming that $x_0 \geq \frac{1}{\sqrt{h(1-\delta)+0.5}}$, we have $     \nabla f(x_0) = \frac{1}{\sqrt{h(1-\delta)+\tfrac{1}{2}}}$. Choosing $d_0 = (1-\delta)\nabla f(x_0)$, which clearly satisfies the definition of relative $\delta$-inexactness, iterate $x_1$ is
    
    \begin{equation*}
        x_1 = x_0 - \frac{h(1-\delta)}{\sqrt{h(1-\delta)+\tfrac{1}{2}}}.
    \end{equation*}
    
    Now choosing $x_0$ such that $x_1 = \frac{1}{\sqrt{h(1-\delta)+\tfrac{1}{2}}}$, i.e. $x_0 = \frac{h(1-\delta)}{\sqrt{h(1-\delta)+\tfrac{1}{2}}} + \frac{1}{\sqrt{h(1-\delta)+\tfrac{1}{2}}}$, 
    it is straightforward to check that $f(x_0)-f(x_\ast) = 1$ and we have shown that
    
    \begin{equation*}
    \begin{aligned}
        \|g_1\|^2 &= \frac{1}{h(1-\delta)+\tfrac{1}{2}}\\
            &= \frac{f(x_0)-f(x_\ast)}{h(1-\delta)+\tfrac{1}{2}}.
    \end{aligned}
    \end{equation*}
    
    Since this matches the convergence rate from Theorem~\ref{theorem::full_th_inexact_one_step} exactly, this proves tightness on the left regime.
    
    \item[Right regime $h\geq \frac{3\delta+2-\sqrt{4-3\delta^2}}{2\delta(\delta+1)}$]
    We now turn to the following univariate quadratic function
    \begin{equation*}
        f(x)=\frac{x^2}{2},
    \end{equation*}
    whose minimizer is $x=0$, with corresponding function value $f(0)=0$.
    Its gradient is $\nabla f(x)=x$. Furthermore choosing $d_0=(1+\delta)\nabla f(x_0)=(1+\delta)x_0$, satisfying the definition of relative $\delta$-inexactness leads to iterate $x_1$ such that
    \begin{equation*}
    \begin{aligned}
        x_1 &= x_0 - hd_0\\
            &= x_0 - h x_0(1+\delta)\\
            &= x_0 (1-h(1+\delta)) 
    \end{aligned}
    \end{equation*}

    To get $f(x_0)-f(x_\ast)=f(x_0)=1$, we thus need $x_0^2=2$. Iterate
    $x_1$ can thus be computed straightforwardly, and we get
    \begin{equation*}
        \begin{aligned}
            \|g_1\|^2=x_1^2&=x_0^2 (1-h(1+\delta))^2\\
            &= 2 (1-h(1+\delta))^2\\
            &= 2 (1-h(1+\delta))^2(f(x_0)-f(x_\ast))
        \end{aligned}
    \end{equation*}
    again matching the theoretical bound in Theorem \ref{theorem::full_th_inexact_one_step}. Tightness is proven.
    \end{description}
\end{proof}

The above result closely mirrors what occurs in the exact gradient descent setting, where the worst-case for the left and right regimes also happens respectively for a Huber and a quadratic function (see e.g. \cite{rotaru2024exact}). Interestingly, the worst-case inexact directions satisfy a very specific structure for those left and right regimes,  
\[
d_0 = (1 \pm \delta) \nabla f(x_0).
\]  
This observation leads to the following remark, which offers an intuitive interpretation of the worst-case behavior.

\begin{remark}
    When the inexact direction takes the form \( d_0 = (1 \pm \delta) \nabla f(x_0) \), the update step of inexact gradient descent can be reinterpreted as an exact gradient step with a miscalibrated stepsize
    \begin{equation*}
    \begin{aligned}
        x_1 &= x_0 - h d_0 \\
        &=  x_0 - h_{\text{eff}} \nabla f(x_0), \quad \text{where } h_{\text{eff}} = h(1 \pm \delta).
    \end{aligned}
    \end{equation*}
    In other words, the inexactness effectively scales the stepsize, leading to either a too conservative or too aggressive update.
    
    This provides a useful perspective: in the worst-case scenarios, relative inexactness does not merely distort the direction of descent, it acts as a perturbation of the stepsize itself. As such, the worst possible errors under a relative inexactness model can be interpreted as poor stepsize tuning, either undershooting or overshooting the optimal update.

    This phenomenon is illustrated in Figure \ref{fig:inexact_on_wc_functions}, which displays the behavior of exact and inexact gradient descent on the 1D worst-case functions identified in Theorem \ref{th::tight_inexact}.
\end{remark}

\begin{figure}[H]
    \centering
    \begin{subfigure}{0.49\textwidth}
        \includegraphics[width=\textwidth]{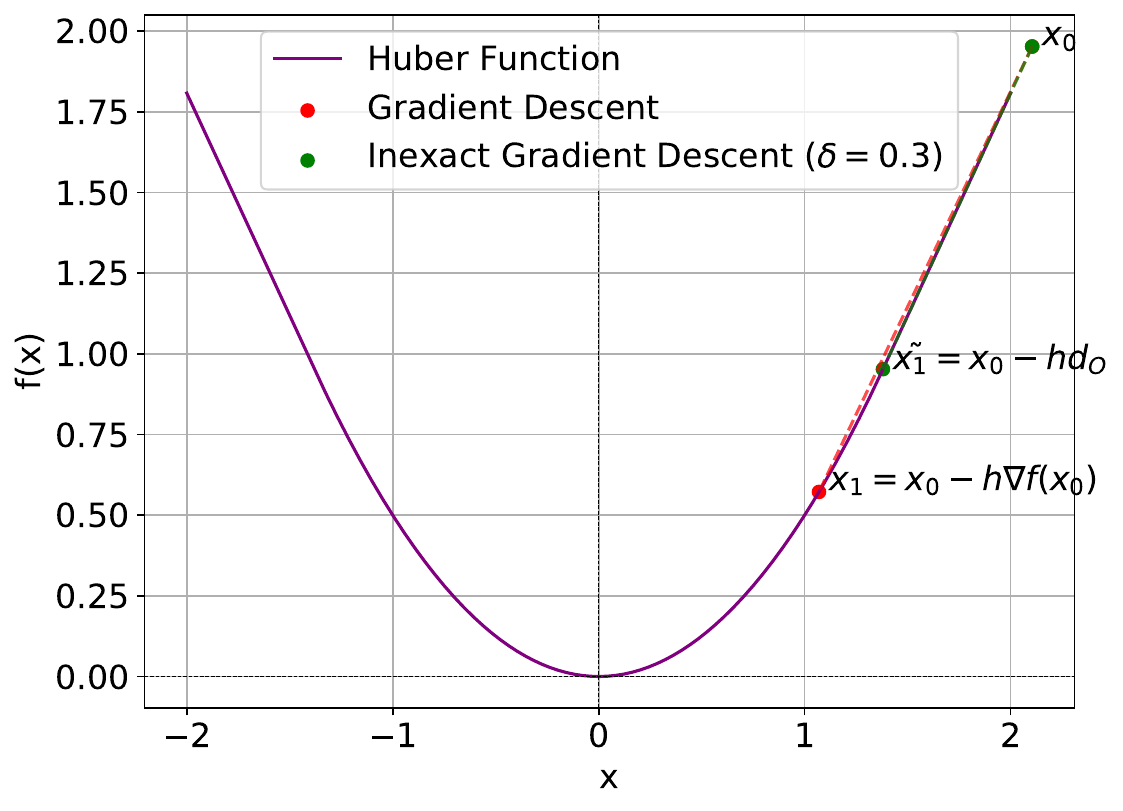}
        \caption{Left (undershoot) regime ($h=0.75$)}
    \end{subfigure}
    \begin{subfigure}{0.49\textwidth}
        \includegraphics[width=\textwidth]{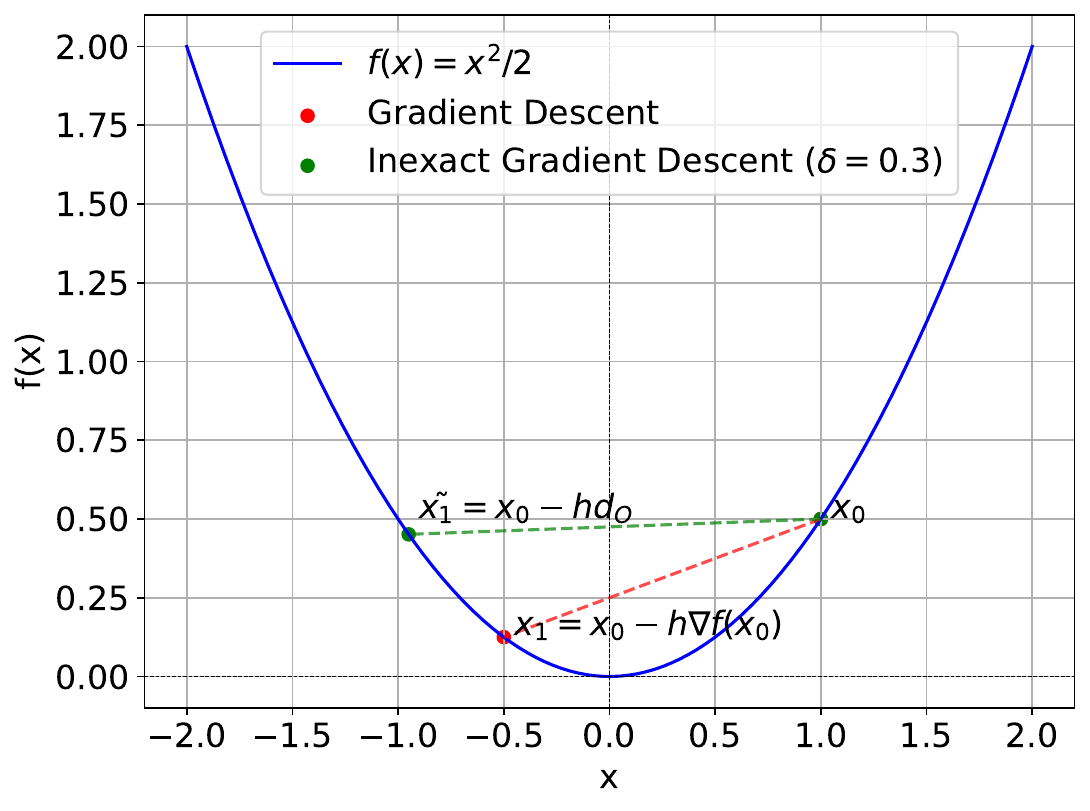}
        \caption{Right (overshoot) regime ($h=1.5$)}
    \end{subfigure}
    \caption{Exact and inexact ($\delta=0.3$) gradient descent on worst-case functions for left and right regimes.}
    \label{fig:inexact_on_wc_functions}
\end{figure}
From these plots, we observe a clear pattern: in the left regime, where the stepsize is already small, the worst-case effect of relative inexactness is to make the step even shorter---further slowing down convergence. Conversely, in the right regime, where the stepsize is large, the worst-case is to amplify it even more---leading to larger overshooting.

This interpretation links naturally to earlier insights on gradient descent sensitivity to stepsize, such as those discussed around Figure 3 in \cite{taylor2017smooth}. Moreover, similar patterns of behavior seem to emerge when the method is run over multiple iterations, suggesting that the link between relative inexactness and stepsize miscalibration persists across time and may be a core feature of such methods.
\subsection{Tightness of the rate in the intermediate regime}
We now focus on the intermediate regime, where the situation is more subtle. Indeed, no univariate function seems able to match the bound. As we explain below, this appears to be a structural limitation: achieving the worst-case rate requires geometric flexibility that only arises in dimension two or higher. This leads us to formulate the following conjecture:

\begin{conjecture}[Tightness of intermediate regime]
    The intermediate regime in Theorem \ref{theorem::full_th_inexact_one_step} is tight. Furthermore, the corresponding worst-case function is bivariate.
\end{conjecture}

This conjecture is supported by strong numerical evidence obtained by solving the PEP. More specifically, for any $\delta\in (0,1)$, we compute the numerical tight worst-case rate over a fine grid of stepsizes $h$ belonging to the intermediate regime, and observe that all of them perfectly match the bound of Theorem \ref{theorem::full_th_inexact_one_step}.
In addition, we observe that a specific two-dimensional construction achieves a \emph{worse} convergence rate than any univariate candidate. This suggests that the true worst-case behavior in this regime fundamentally requires two dimensions.

To make this insight precise, assume for contradiction that the worst case is realized by a univariate function. Then, without loss of generality, we may take
\[
\nabla f(x_0) = \begin{pmatrix} a \\ 0 \end{pmatrix}, \quad d_0 = \begin{pmatrix} \tilde{a} \\ 0 \end{pmatrix},
\]
and saturate the relative inexactness condition, leading to \( \|d_0 - \nabla f(x_0)\| = \delta \|\nabla f(x_0)\| \), so that \( \tilde{a} \in [(1-\delta)a, (1+\delta)a] \). 
Now, consider instead the bivariate construction
\[
\nabla f(x_0) = \begin{pmatrix} a \\ 0 \end{pmatrix}, \quad
d_0 = \begin{pmatrix} (1-\delta^2)a \\ \delta \sqrt{1-\delta^2}a \end{pmatrix},
\]
which also saturates the inexactness condition \eqref{ine}, and enforces that the error $d_0-\nabla f(x_0)$ is orthogonal to the gradient, i.e. $\langle d_0 - \nabla f(x_0), \nabla f(x_0) \rangle = 0$. Orthogonality is possible due to some rotational freedom that is not available in one dimension. We then maximize the final gradient norm \( \|\nabla f(x_1)\|^2 \) over all admissible values of \( \nabla f(x_1) = (u, v)^\top \). Numerical resolution of the PEP confirms that letting both \( u \) and \( v \) vary leads to a strictly larger worst-case value than in the univariate case (where necessarily \( v = 0 \)).

This phenomenon is illustrated in Figure \ref{fig:1D_2D_WC_fct}, which compares the theoretical rate with the worst-case rates achieved by 1D and 2D constructions\footnote{Solving a PEP in fixed dimension requires a non-convex formulation, see e.g. \cite{das2024branch}.}. We observe that the proposed 2D construction matches the theoretical bound specifically at the optimal stepsize, leading to a numerical proof that the bound is tight in that case, whereas the \emph{worse} 1D constructions fall short. This provides strong support for the second part of the conjecture.

\begin{figure}[H]
    \centering
    \begin{subfigure}{0.32\textwidth}
        \includegraphics[width=\textwidth]{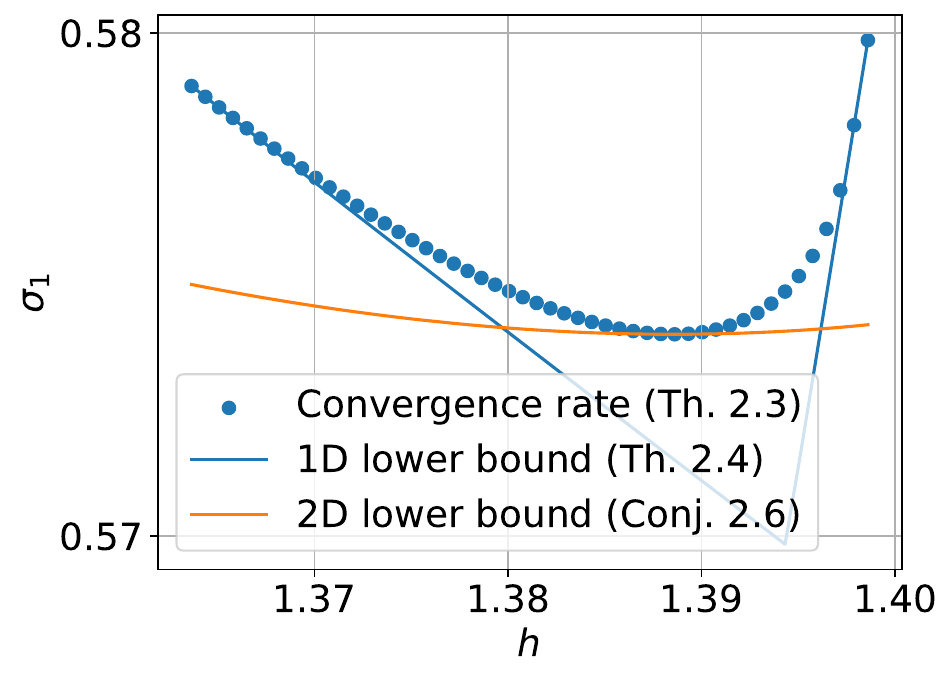}
        \caption{$\delta = 0.1$}
    \end{subfigure}
    \begin{subfigure}{0.32\textwidth}
        \includegraphics[width=\textwidth]{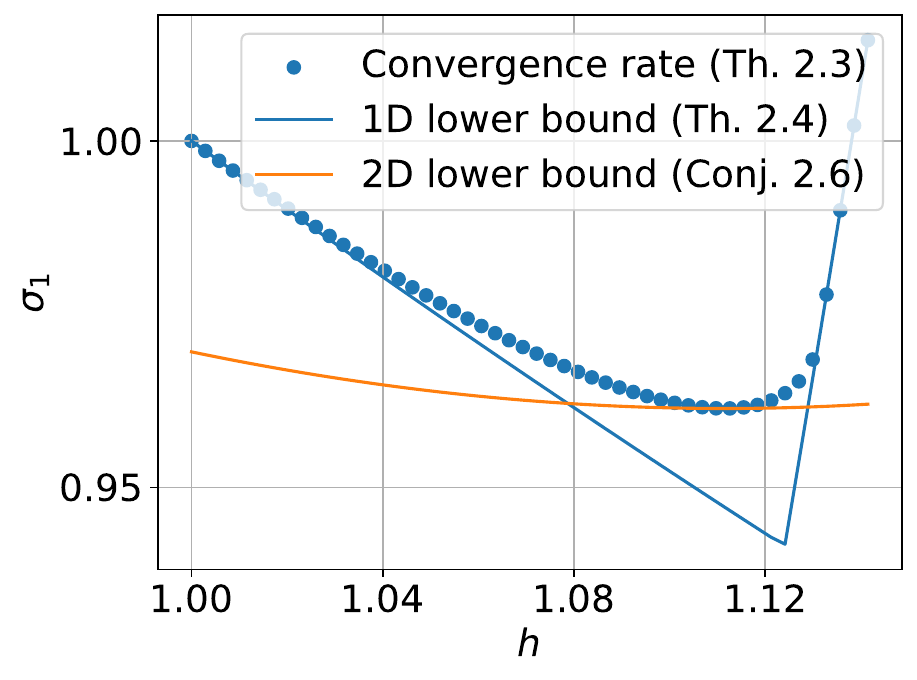}
        \caption{$\delta = 0.5$}
    \end{subfigure}
    \begin{subfigure}{0.32\textwidth}
        \includegraphics[width=\textwidth]{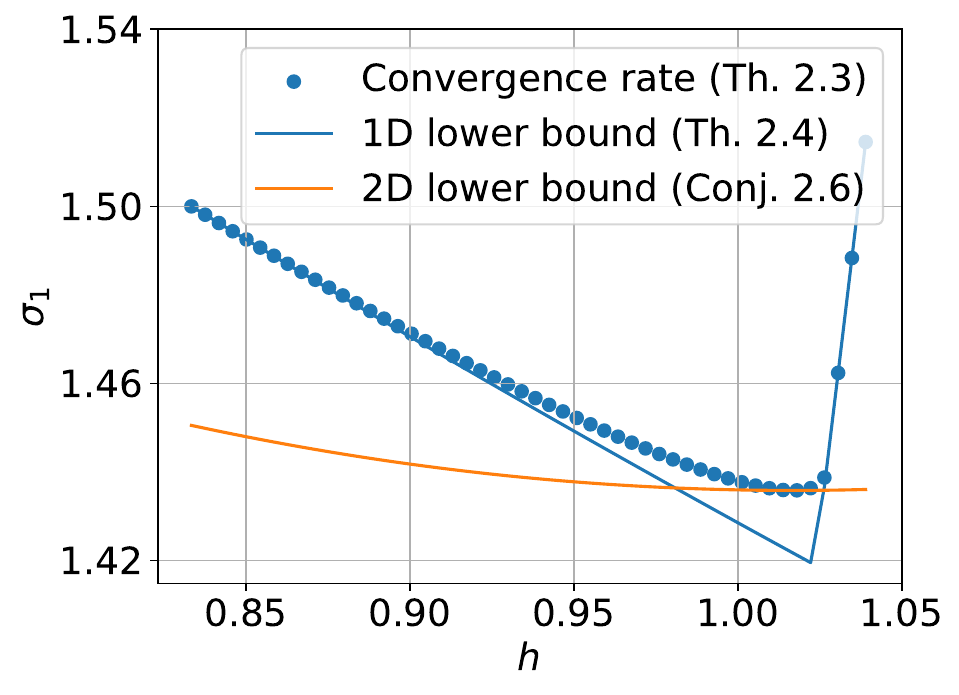}
        \caption{$\delta = 0.8$}
    \end{subfigure}
    \caption{\centering Convergence rate of intermediate regime from Theorem 
    \ref{theorem::full_th_inexact_one_step} versus 1D and 2D worst-case constructions for different values of $\delta$.}
    \label{fig:1D_2D_WC_fct}
\end{figure}

The key insight here is geometric: in the intermediate regime, the worst-case error vector \( d_0 - \nabla f(x_0) \) is no longer colinear with the true gradient. Since the relative inexactness condition \eqref{ine} constrains this error to lie within a ball of radius \( \delta \|\nabla f(x_0)\| \), the worst case occurs when the error lies on the boundary of the ball. This gives rise to a rotational degree of freedom that allows the direction error to “turn” around the gradient, leading to greater degradation in performance. This phenomenon is illustrated in Figure \ref{fig::turning_error}.

\begin{figure}[H]
    \centering
    \includegraphics[width=0.7\linewidth]{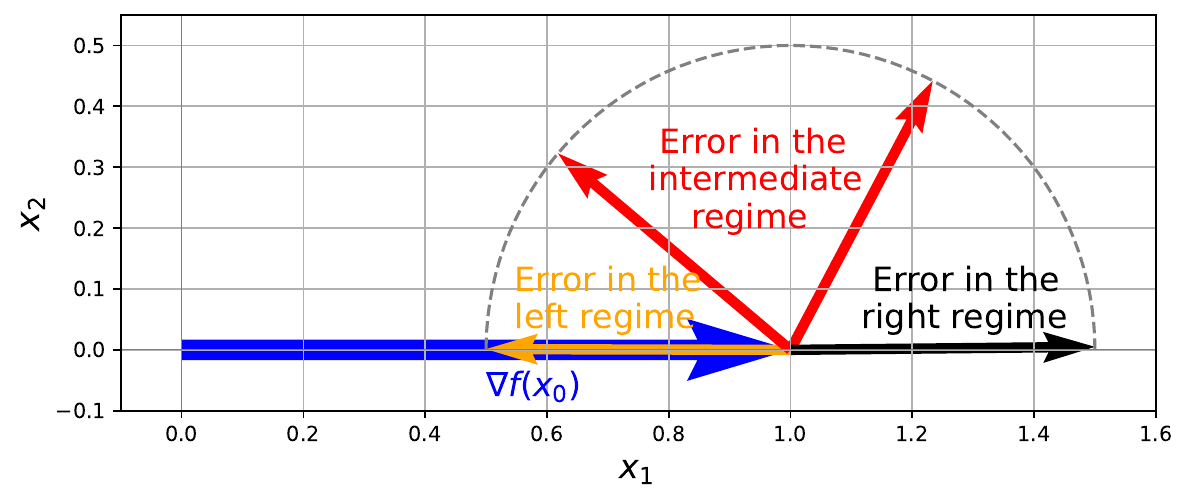}
    \caption{\centering Worst-case direction error \( d_0 - \nabla f(x_0) \) rotating around the gradient in the intermediate regime (example with normalized gradient and \( \delta=0.5 \)).}
    \label{fig::turning_error}
\end{figure}

Moreover, numerical evidence from the PEP experiments suggests that the orthogonality condition discussed above is met \emph{precisely at the optimal (best) stepsize}, i.e.,
\[
\langle d_0 - \nabla f(x_0), \nabla f(x_0) \rangle = 0 \quad \text{at the best stepsize}
\]
meaning that the worst-case error vector is orthogonal to the true gradient only when the method is tuned to its best stepsize. This condition is satisfied by the bivariate construction above and seems to be a consistent feature observed in our numerical worst-case PEP solutions. However, due to the algebraic complexity of the problem, we have not derived an explicit analytical worst-case function in this regime. Investigating whether this orthogonality condition always corresponds to the best stepsize, and how it depends on \(\delta\), is an interesting direction for future work.

\section{Generalization to several steps}\label{sec::several_steps}

Extending our analysis to multiple iterations reveals new insights, but also new challenges. While the exact convergence rate for a single step of relatively inexact gradient descent already required solving a non-trivial cubic equation, the situation becomes even more complicated when analyzing the behavior over several steps, especially in the intermediate regime of stepsizes. An exact characterization of the worst-case rate for an arbitrary number of iterations is unlikely to be obtainable analytically. However, we can derive useful and informative \emph{upper and lower bounds} that capture the behavior of the method over several iterations.

\subsection{Upper bound on the rate for several steps}

\begin{theorem}[Upper bound on the convergence rate for $N$ steps of inexact gradient descent on smooth convex functions]\label{theorem::full_th_inexact_N_steps}
    Algorithm \ref{algo::inexact_gradient} applied to a convex $L$-smooth function $f$ with an inexact gradient with relative inexactness $\delta\in(0,1)$, 
    started from iterate $x_0$ with 
    a constant stepsize $h\in[0,\frac{2}{1+\delta}]$,  generates iterates satisfying
    \[
        \frac{1}{L} \min_{k\in\{1,\cdots,N\}}\|\nabla f(x_k)\|^2 \leq \tilde{C}_N(h, \delta) (f(x_0) - f(x_\ast)),
    \]
    where $\tilde{C}_N(h, \delta)$ is given by the continuous function
        \small
\[
\tilde{C}_N(h, \delta) =
\left\{
\begin{array}{lll}
   \frac{1}{Nh(1-\delta)+\tfrac{1}{2}}, & \text{if } h \in \left[0, \frac{3}{2(1+\delta)}\right) & \text{(left regime)} \\
    \frac{2\tilde{\lambda}}{N(h\tilde{\lambda}^2+2(h-1)\tilde{\lambda}+h-1)+\tilde{\lambda}}, & \text{if } h \in \left[\frac{3}{2(1+\delta)}, \frac{3\delta+2-\sqrt{4-3\delta^2}}{2\delta(\delta+1)}\right]
        & \text{(intermediate regime)} \\
    \frac{2}{N((1-h(1+\delta))^{-2})-(N-1)}, & \text{if } h \in \left(\frac{3\delta+2-\sqrt{4-3\delta^2}}{2\delta(\delta+1)}, \frac{2}{1+\delta}\right]
        & \text{(right regime)}
\end{array}
\right.
\]
\normalsize
\end{theorem}
\begin{proof} 
    The first rate from Theorem \ref{theorem::full_th_inexact_one_step} can be generalized in a very straightforward manner. We use it at each iteration to write 
    \[ \frac{1}{L} \|\nabla f(x_{k+1})\|^2 \leq C(h, \delta) (f(x_k) - f(x_{k+1})),\]
     with $C(h,\delta)$ defined in the aforementioned theorem, and summing it from $k=0$ to $N-1$ leads to
        \begin{equation*}
        \begin{aligned}
            \sum_{k=0}^{N-1}\|\nabla f(x_{k+1})\|^2 &\leq \sum_{k=0}^{N-1}C(h, \delta) (f(x_k) - f(x_{k+1}))\\
            &=C(h, \delta) (f(x_0)-f(x_N))
            \end{aligned}
        \end{equation*}
        Note that the sum of $N$ squared gradient norms appearing in the left-hand side is always greater than or equal to $N$ times its minimum term, i.e. we have 
         $\sum_{k=0}^{N-1} \|\nabla f(x_k)\|^2 \ge N \min_{k\in\{1,\cdots,N\}} \|\nabla f(x_k)\|^2$. Following the argument from the end of the proof of Theorem \ref{th::exact_convergence_rate} allows to replace $f(x_N)$ and $C_N(h,\delta)$ by $f(x_\ast)$ and $\tilde{C}_N(h,\delta)$, and yields the stated result.
\end{proof}
This upper bound can be visualized in Figure \ref{fig:lower_upper_several_steps} for different inexactness levels $\delta$ and different numbers of steps $N$, together with the lower bound from the next subsection and the exact value of the rate computed numerically by a PEP.

\subsection{Lower bound on the rate for several steps}
We now provide lower bounds on the convergence rate of inexact gradient descent, using explicit smooth convex functions whose convergence rates after $N$ steps are close to the upper bound derived above.

\begin{theorem}[Lower bound on the convergence rate for $N$ steps of inexact gradient descent on smooth convex functions]\label{theorem::lower_inexact_N_steps}
   When applying any number of steps $N \in \mathbb{N}$ of Algorithm \ref{algo::inexact_gradient} with an inexact gradient with relative inexactness $\delta\in(0,1)$ with a constant stepsize $h\in[0,\frac{2}{1+\delta}]$, there exists an $L$-smooth convex function $f$ and an initial point $x_0$ such that iterates satisfy
    \begin{equation*}
        \max\left\{\frac{1}{Nh(1-\delta)+\tfrac{1}{2}},2(1-h(1+\delta))^{-2N}\right\} (f(x_0)-f(x_\ast))\leq\frac{1}{L} \|\nabla f(x_N)\|^2
    \end{equation*}
\end{theorem}
\begin{proof}
Consider the following two $1$-smooth convex functions
\begin{equation*}
f_1(x) =
    \left\{
    \begin{array}{rl}
    \tfrac{1}{\sqrt{Nh(1-\delta)+\tfrac{1}{2}}} |x| - \tfrac{1}{2Nh(1-\delta)+\tfrac{1}{2}} ~~\text{, when } |x| \geq \tfrac{1}{\sqrt{Nh(1-\delta)+\tfrac{1}{2}}} \\
    \frac{x^2}{2}~~\text{, when } |x| < \tfrac{1}{\sqrt{Nh(1-\delta)+\tfrac{1}{2}}}
    \end{array}
    \right.
\end{equation*}
and $f_2(x)=\frac{x^2}{2}$, i.e. a Huber and a quadratic function. The stated result then follows using an argument similar to that in Theorem \ref{th::tight_inexact}, adapted to the $N$-steps case. 
\end{proof}

The upper and lower bounds provided by Theorems \ref{theorem::full_th_inexact_N_steps} and \ref{theorem::lower_inexact_N_steps} are illustrated in Figure \ref{fig:lower_upper_several_steps} for several values of inexactness level $\delta$ and numbers of iterations $N$. The PEP-based numerical worst-case rates are plotted alongside the analytical bounds. 
\begin{figure}[H]
    \centering
\begin{subfigure}{0.32\textwidth}
  \includegraphics[width=\linewidth]{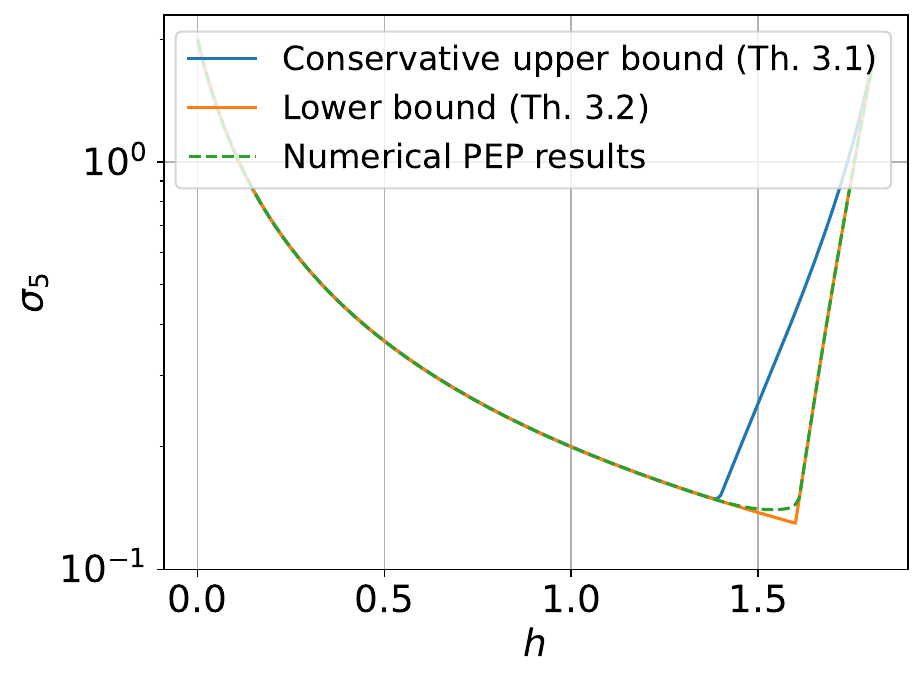}
  \caption{\centering$\delta=0.1$, $N=5$}
\end{subfigure}\hfil %
\begin{subfigure}{0.32\textwidth}
  \includegraphics[width=\linewidth]{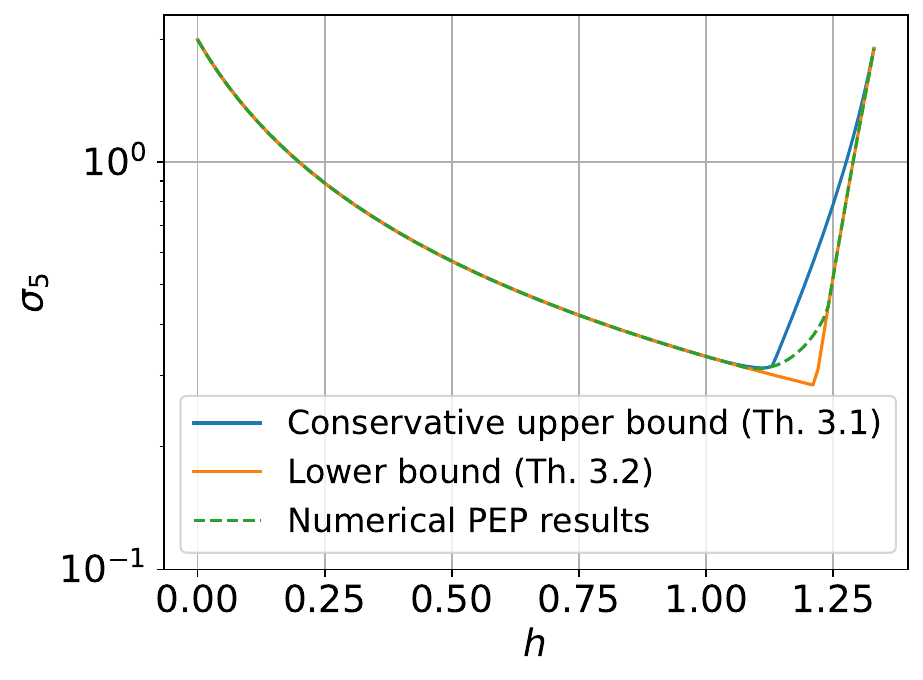}
  \caption{\centering$\delta=0.5$, $N=5$}
\end{subfigure}\hfil 
\begin{subfigure}{0.32\textwidth}
  \includegraphics[width=\linewidth]{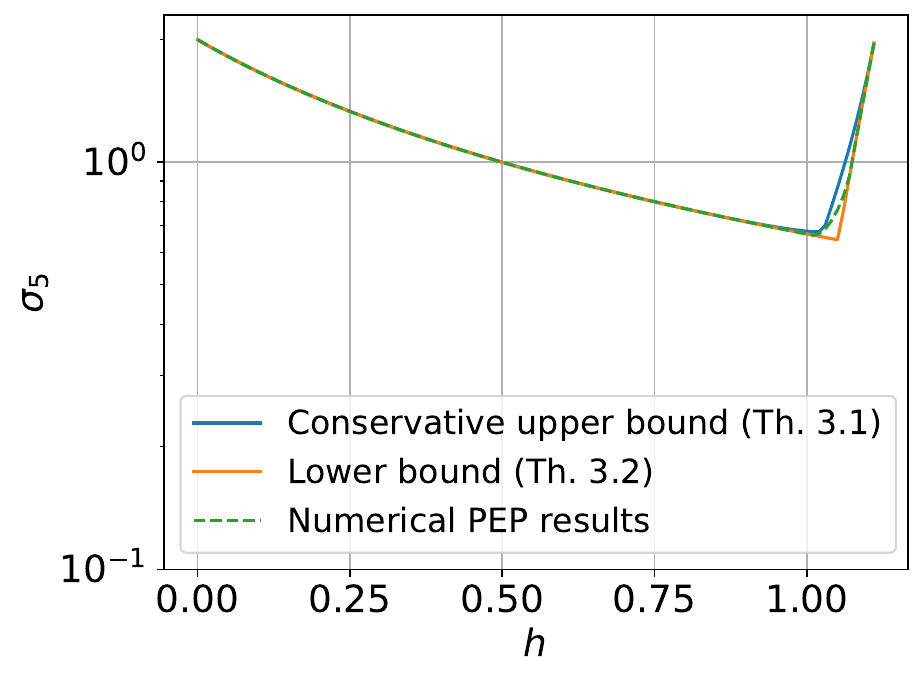}
  \caption{\centering$\delta=0.8$, $N=5$}
\end{subfigure}

\medskip
\begin{subfigure}{0.32\textwidth}
  \includegraphics[width=\linewidth]{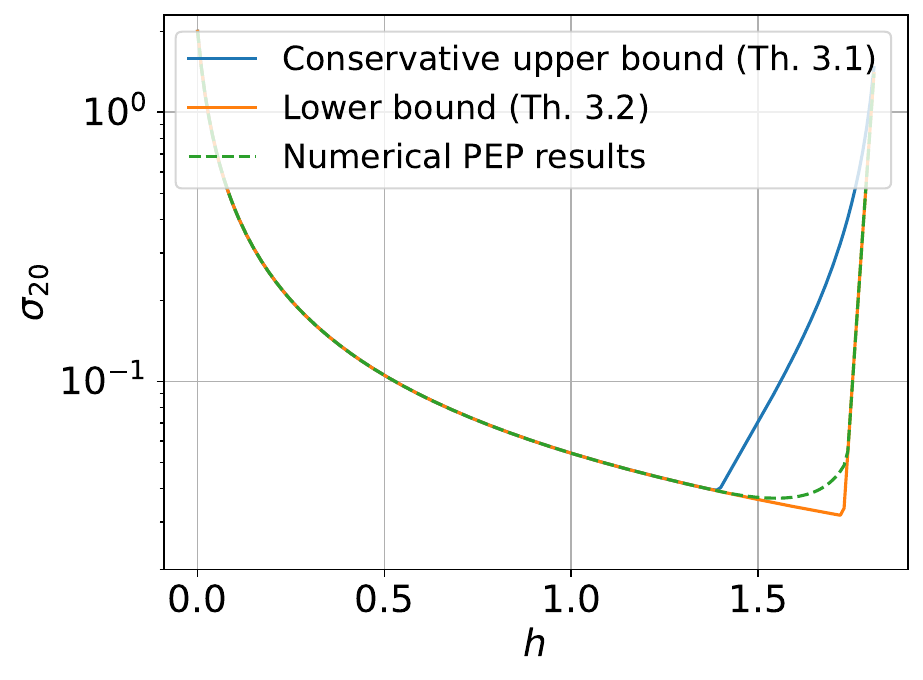}
  \caption{\centering$\delta=0.1$, $N=20$}
\end{subfigure}\hfil 
\begin{subfigure}{0.32\textwidth}
  \includegraphics[width=\linewidth]{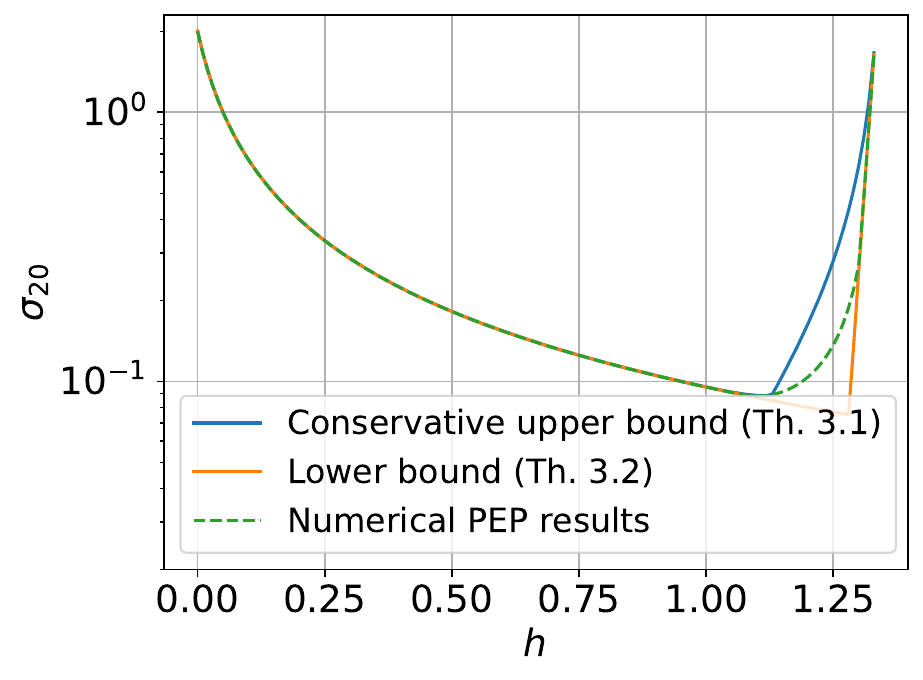}
  \caption{\centering$\delta=0.5$, $N=20$}
\end{subfigure}\hfil 
\begin{subfigure}{0.32\textwidth}
  \includegraphics[width=\linewidth]{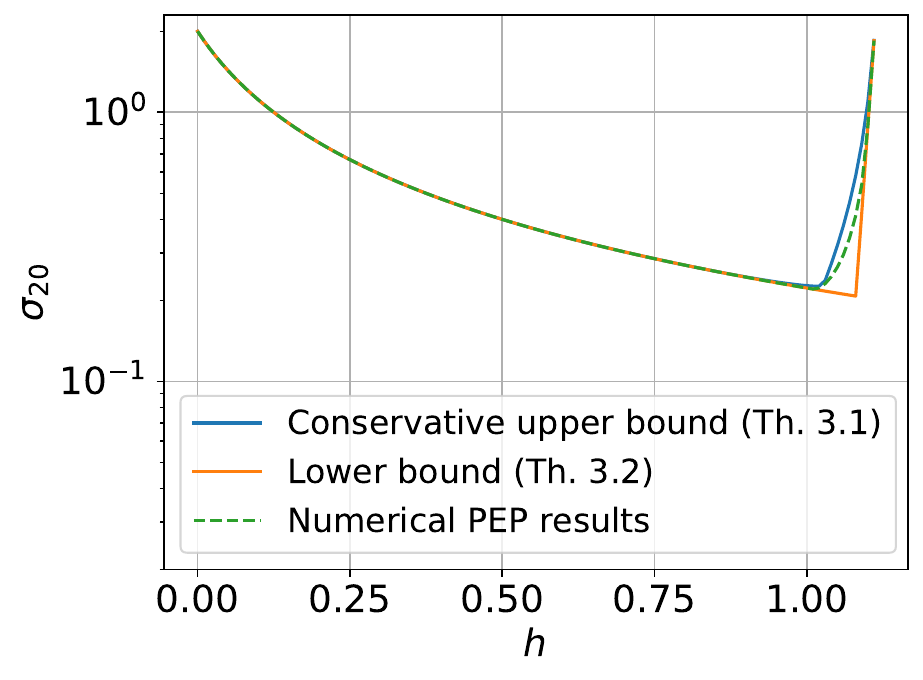}
  \caption{\centering$\delta=0.8$, $N=20$}
\end{subfigure}
\caption{\centering Lower and upper bounds and exact value of the convergence rate for different inexactness levels $\delta$ and different number of steps $N$.}
\label{fig:lower_upper_several_steps}
\end{figure}

In the left regime, upper and lower bounds coincide exactly, thereby establishing the tightness of the convergence rate in this case. In contrast, the bounds do not match in the other two regimes; however, they still offer valuable insights into the worst-case behavior of inexact gradient descent when multiple iterations are performed. In particular, we can make an important remark that partially tightens the convergence rate in a subinterval of the right regime.

\begin{remark}
    In the right stepsize regime \( h \in \left]\frac{3\delta+2-\sqrt{4-3\delta^2}}{2\delta(\delta+1)}, \frac{2}{1+\delta}\right[ \), the upper bound
    \[
        \frac{2}{N((1-h(1+\delta))^{-2})-(N-1)}
    \]
    is valid. Moreover, it can be shown that a sharper upper bound, matching the lower bound
    \[
        2(1 - h(1+\delta))^{-2N},
    \]
    also holds for all \( h \geq \tilde{h}(N) \), for some threshold \( \tilde{h}(N) \) belonging to the right regime. This leads to a tight convergence rate for several steps in the subinterval \( [\tilde{h}(N), \frac{2}{1+\delta}) \).
        However, the threshold \( \tilde{h} \) is defined implicitly as the root of a polynomial equation whose degree increases with $N$, making it difficult to express in closed form.  

    As an illustration, Figure \ref{fig:tight_bound_transition} shows how the numerically determined value of \( \tilde{h}(N) \) evolves with \( N \) for a few values of $\delta$.  The Python code used for these experiments, based on the \texttt{SymPy} package but used here purely for numerical purposes, is available online.
    
    \begin{figure}[H]
    \centering
    \includegraphics[width=0.5\textwidth]{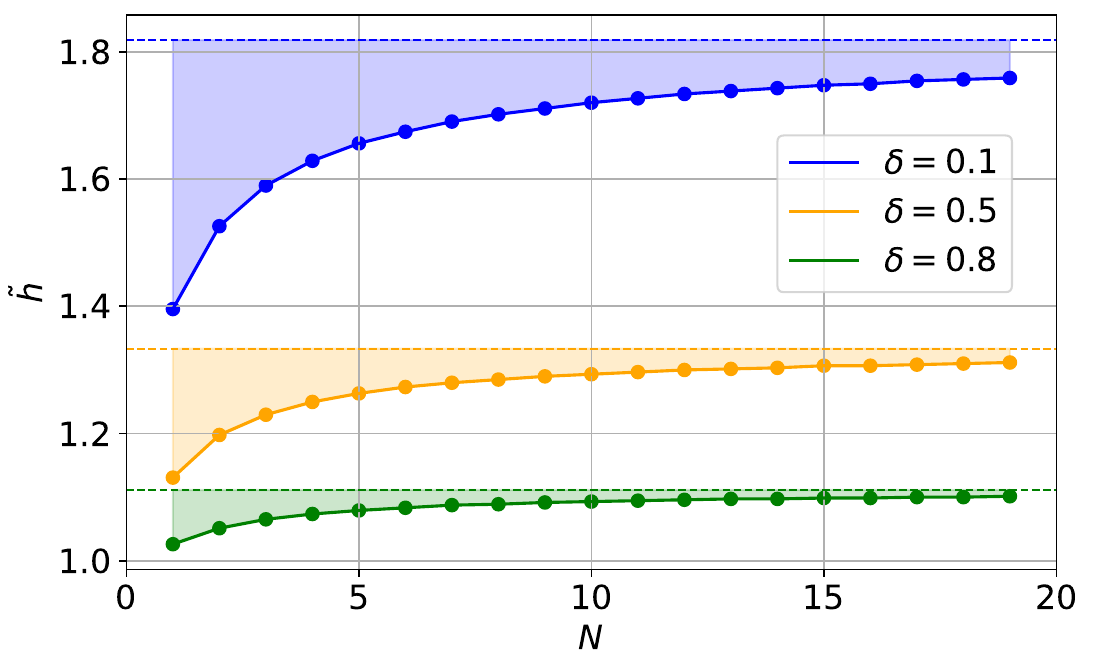}
    \caption{\centering Threshold $\tilde{h}(N)$  as a function of $N$ for different values of $\delta$ (dots). Dashed lines indicate the maximum stepsize $2/(1+\delta)$, and the shaded regions highlight the interval between $\tilde{h}$ and this threshold.}
    \label{fig:tight_bound_transition}
\end{figure}

\end{remark}

\subsection{Approximation of the optimal stepsize}

We now investigate the optimal stepsize $h_{\text{opt}}(N,\delta)$, defined as the value of $h$ that minimizes the (numerical) exact PEP worst-case bound for a given $\delta$ and $N$. Figure \ref{fig::best_h_N} plots the numerical value of $h_{\text{opt}}(N,\delta)$ (solid lines), and the end of the one-step intermediate regime in the inexact case (dashed lines, see Theorem \ref{theorem::full_th_inexact_one_step}).
\begin{figure}[H]
    \centering
    \includegraphics[width=0.6\linewidth]{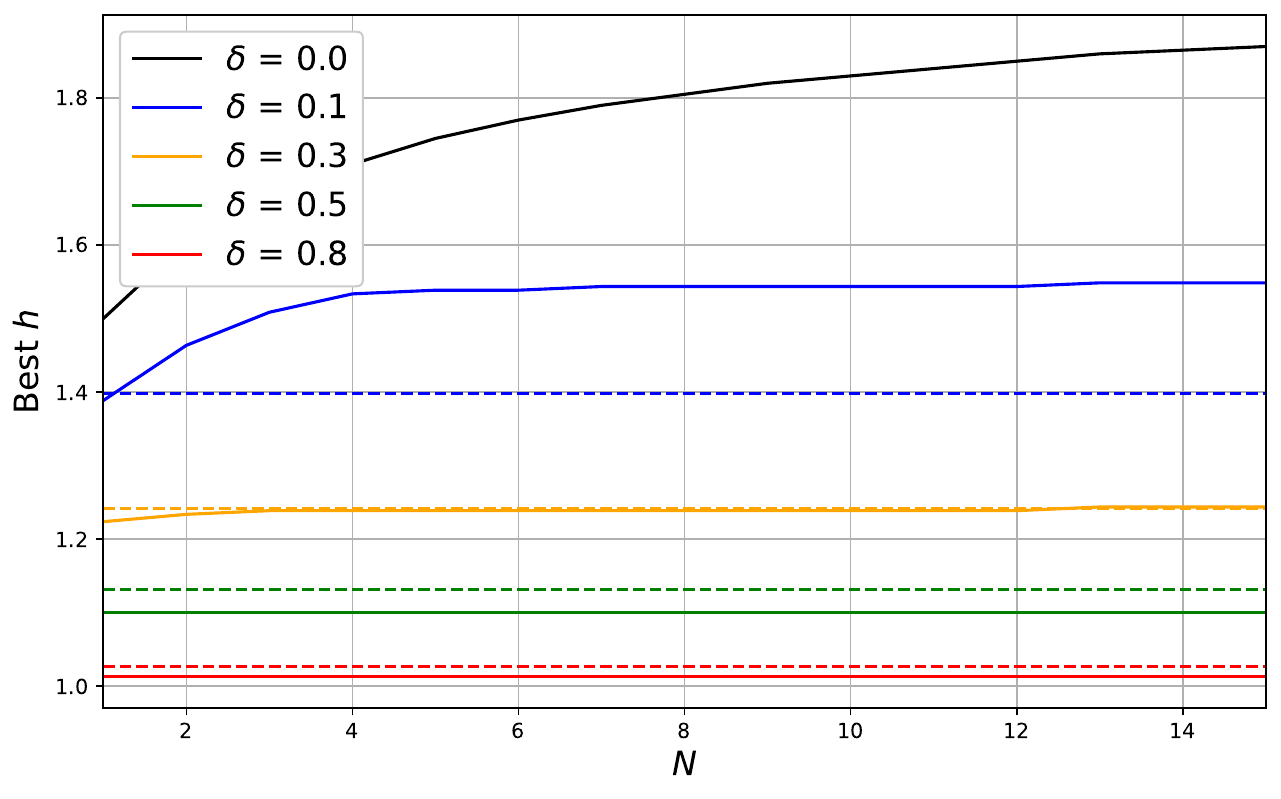}
    \caption{\centering Numerical $h_{\text{opt}}$ depending on the number of iterations $N$ for different levels of inexactness $\delta$ (solid lines) and corresponding end of 1-step intermediate regimes (dashed lines).}
    \label{fig::best_h_N}
\end{figure}
We observe two distinct behaviors depending on the inexactness level $\delta$:
\begin{itemize}
    \item For small values of $\delta$, the optimal stepsize increases with $N$, much like in the exact case ($\delta=0$), in which it tends toward $2$.
    \item For moderate to large values of $\delta$, the optimal stepsize stabilizes, becoming nearly independent of $N$. This indicates a transition to a regime where increasing the number of steps does not justify taking larger stepsizes.
\end{itemize}

Figure \ref{fig::best_rate_vs_end_middle} confirms that choosing as stepsize the right end of the intermediate regime, namely $h=\frac{3\delta+2-\sqrt{4-3\delta^2}}{2\delta(\delta+1)}$,  yields a convergence rate extremely close to the optimal one (at $h_\text{opt}$) for a wide range of values of $N$, especially when $\delta \geq 0.2$. This simple formula, independent from the number of iterations $N$, is thus a highly effective and practical stepsize choice in the inexact regime.

\begin{figure}[H]
    \centering
    \includegraphics[width=0.6\linewidth]{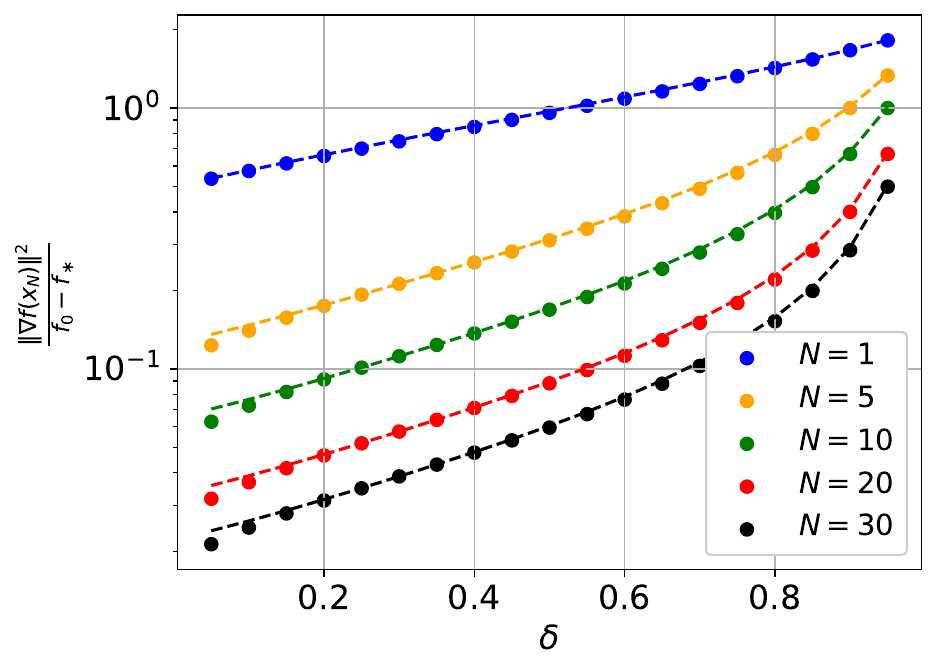}
    \caption{\centering Numerical rate at $h_{\text{opt}}$ (solid dots) and at its approximation by $\frac{3\delta+2-\sqrt{4-3\delta^2}}{2\delta(\delta+1)}$ (dashed lines) for different number of iterations.}
    \label{fig::best_rate_vs_end_middle}
\end{figure}

This sharply contrasts with the exact case, where stepsize $h$ must grow with $N$ to maintain optimality. In the relatively inexact setting, large values of the stepsize can lead to excessive amplification of errors, which explains the stabilization of $h_{\text{opt}}$ as $N$ increases.

\section{Conclusion}\label{sec::conclusion}
In this work, we conducted a detailed analysis of the constant stepsize gradient descent method under relative inexactness. Our contributions are twofold.

First, in Section \ref{sec::one_step}, we established a tight theoretical characterization of the one-step behavior of relatively inexact gradient descent on smooth convex functions. We derived an explicit worst-case convergence rate $\sigma_1$, composed of three different stepsize regimes, and proved its tightness, either analytically or numerically. We observed that inexactness does not significantly deteriorate the convergence rate, but noted that smaller stepsizes should be used compared to the exact case. We also discussed that characterizing the optimal stepsize analytically is hard, due to difficult symbolic computations.

Secondly, in Section \ref{sec::several_steps}, we extended the analysis to multiple iterations, providing both upper and lower bounds on the worst-case convergence rate. This was achieved through a combination of analytical reasoning and computer-aided proofs using the Performance Estimation Problem (PEP) framework. The provided lower and upper bounds were often close to each other, providing insight about the method's behavior when performing several iterations. We also observed that the best stepsize according to the obtained worst-case convergence rates is not very sensitive to the number of iterations, highlighting a significant difference with the exact case. We also gave a simple analytical approximation of the optimal stepsize that can deliver near-optimal performance across any number of iterations.

Beyond these results, our study highlights the intrinsic complexity of analyzing the exact worst-case behavior of relatively inexact methods. In particular, the intermediate regime involves solving a high-degree polynomial system, whose structure resists symbolic simplification. This suggests that the difficulty is not only technical but may reflect a deeper hardness of the underlying worst-case analysis. However, it is also possible that this hardness results from some choices made in the setup of our analysis (such as the performance criteria used to measure convergence). Nevertheless, such challenges underscore the importance of combining symbolic analysis with numerical and computer-assisted tools to gain insights into complex inexact optimization settings.

Looking ahead, several promising research directions emerge. First, our numerical results suggest that the worst-case error vector becomes orthogonal to the true gradient precisely when the stepsize is optimally tuned. Formalizing this observation and understanding how this orthogonality condition depends on the inexactness level $\delta$ remains an open theoretical challenge. Second, our analysis reveals a shift in optimal algorithm design under relative gradient noise in the smooth convex setting: taking many small steps is preferable to using large steps. Extending this insight to the stochastic case, where inexactness arises from sampling noise, is a natural and important next step. In particular, it would be valuable to investigate whether similar stepsize tuning principles apply when stochasticity and relative inexactness interact. Initial progress in this direction has been made recently by using the PEP methodology to analyze stochastic gradient descent without variance assumptions \cite{cortild2025new}. \rev{This analysis has been expanded even more recently in \cite{rubbens2025computer}, but the full picture remains to be explored.}

\clearpage
\bibliography{my_biblio}

\end{document}